\documentclass[12pt]{article}
\usepackage{graphicx}
\usepackage{amsmath}
\usepackage{amssymb}
\usepackage{theorem}

\numberwithin{equation}{section}

\newtheorem{thm}{Theorem}[section]
\newtheorem{lemma}[thm]{Lemma}

\newtheorem{prop}[thm]{Proposition}
\newtheorem{cor}[thm]{Corollary}
{\theorembodyfont{\rmfamily}
\newtheorem{defn}[thm]{Definition}
\newtheorem{example}[thm]{Example}

\newtheorem{rmk}[thm]{Remark}
}

\newcommand{\qed}{\hfill \mbox{\raggedright \rule{.07in}{.1in}}}
 
\newenvironment{proof}{\vspace{1ex}\noindent{\bf
Proof}\hspace{0.5em}}{\hfill\qed\vspace{1ex}}
\newenvironment{pfof}[1]{\vspace{1ex}\noindent{\bf Proof of
#1}\hspace{0.5em}}{\hfill\qed\vspace{1ex}}

\newcommand{\R}{{\mathbb R}}

\newcommand{\C}{{\mathbb C}}
\newcommand{\Z}{{\mathbb Z}}
\newcommand{\D}{{\mathbb D}}

\newcommand{\spec}{\operatorname{spec}}
\newcommand{\BV}{\operatorname{BV}}

\newcommand{\SMALL}{\textstyle}

\title{Decay of correlations for nonuniformly expanding systems with
general return times}

\author{Ian Melbourne \thanks{Department of Mathematics, University of Surrey,
Guildford, Surrey GU2 7XH, UK}
 \and
 Dalia Terhesiu \thanks{
Department of Mathematics, University of Surrey,
Guildford, Surrey GU2 7XH, UK}
}

\date{20 May, 2011.   Updated 24 August, 2011}

\begin{document}

\maketitle

 \begin{abstract}
We give a unified treatment of decay of correlations for nonuniformly expanding systems with a good inducing scheme.
In addition to being more elementary than previous treatments, our
results hold for general integrable return time functions under fairly mild conditions
on the inducing scheme.
 \end{abstract}

  \section{Introduction} 
  \label{sec-intro}

Let $T:X\to X$ be a (noninvertible) measure preserving transformation
with ergodic invariant probability measure $\mu_X$.    
Given $v\in L^1(X)$, $w\in L^\infty(X)$, we define the correlation
function $\rho_{v,w}(n)=\int_X v\,w\circ T^n\,d\mu_X-\int_X v\,d\mu_X
\int_X w\,d\mu_X$.
If $T$ is mixing, then $\rho_{v,w}(n)\to0$ as $n\to\infty$.

\begin{defn}
Let $\mathcal{B}(X)\subset L^1(X)$ denote a collection of observables $v:X\to\R$.
Let $a_n>0$ be a real sequence with $a_n\to0$.
We say that $T$ has {\em uniform decay rate $a_n$} for observables in $\mathcal{B}(X)$
if for every $v\in\mathcal{B}(X)$ there is a constant $C_v>0$ such that
$|\rho_{v,w}(n)|\le C_v |w|_\infty a_n$ for all $w\in L^\infty(X)$.
\end{defn}

We assume the existence of an induced map $F:Y\to Y$, $Y\subset X$,
given by $F(y)=f^{\varphi(y)}(y)$ for some return time
 $\varphi:Y\to\Z^+$.   
(We do not require that $\varphi$ is the first return time to $Y$.)
It is assumed throughout that $\mu$ is an $F$-invariant ergodic
probability measure on $\mu$ and that $\varphi\in L^1(Y)$.   
The measure $\mu_X$ on $X$ is constructed from $\mu$ and $\varphi$
in the standard way (see Section~\ref{sec-tower}).
The idea is to recover decay properties for $T$ from properties
of $F$ and the return tails $\mu(y\in Y:\varphi(y)>n)$.

In this paper, we combine the  method of
operator renewal sequences~\cite{Gouezel04a,GouezelPhD, Sarig02} with dynamical truncation~\cite{M09} to give a particularly
elementary and general treatment of decay of correlations in a much wider
context than the usual Young tower setting~\cite{Young99}.
Moreover, our results are strictly sharper than those obtained in 
the setting of Young towers by the methods of coupling~\cite{Young99},
Birkhoff cones~\cite{Maume01a} and stochastic perturbation~\cite{LiveraniSaussolVaienti99}.

\subsection{Young towers}

Young~\cite{Young98,Young99} considered the case where $T:X\to X$ is an
ergodic nonuniformly expanding local diffeomorphism on a manifold $X$ modelled
by a {\em Young tower}.   In particular, $F:Y\to Y$ is a uniformly expanding
map with good distortion properties with respect to a countable partition
(a so-called Gibbs-Markov map)
and $\varphi$ is constant on partition elements.
Throughout, this subsection, we take 
$\mathcal{B}(X)$ to be the space of (piecewise) H\"older observables.

In the case where $\mu(\varphi>n)$ decays exponentially,
Young~\cite{Young99} obtained exponential decay of correlations.
In the subexponential case, Young~\cite{Young99} proved (amongst other things)
that if $\mu(\varphi>n)=O(1/n^{\beta+1})$, $\beta>0$, then correlations
decay at the rate $a_n=1/n^\beta$.
This result was shown to
be optimal by Sarig~\cite{Sarig02} and Gou\"ezel~\cite{Gouezel04a}.
Gou\"ezel~\cite{GouezelPhD} introduced a very general class of 
{\em convolutive sequences} and proved that if $\mu(\varphi>n)=O(b_n)$
where $b_n$ is convolutive, then decay of correlations holds with
optimal rate $a_n=\sum_{j>n}b_j$.
This includes the cases of stretched exponential decay of correlations 
(Example~\ref{ex-stretch}) and 
{\em polynomially decreasing} sequences (Example~\ref{ex-poly}).

Even in the context of Young towers, we obtain a number of new results.
We mention three of these now.  (The general formulation Theorem~\ref{thm-main} of our result
is somewhat technical and hence delayed until Section~\ref{sec-L}.)

\begin{thm}  \label{thm-poly}
Suppose that $\varphi\in L^{1+\epsilon}(Y)$ for some $\epsilon>0$.
Then for any $p>0$ there exists $\delta>0$, $C>0$ such that
\[
\rho_{v,w}(n)\le C\|v\| |w|_\infty\Bigl\{\sum_{j>\delta n}\mu(\varphi>j)+n\mu(\varphi>\delta n)+O(n^{-p})\Bigr\},
\]
for all $v\in\mathcal{B}(X)$, $w\in L^\infty(X)$, $n\ge1$.
\end{thm}

An immediate consequence of Theorem~\ref{thm-poly} is optimal upper bounds on decay of correlations when
$\mu(\varphi>n)=O(1/n^{\beta+1})$ for $\beta>0$.    More generally, the case 
when $\mu(\varphi>n)$ is dominated by
a regularly varying sequence $\ell(n)/n^{\beta+1}$ also follows from
Theorem~\ref{thm-poly}, as does the even more general situation
where $\mu(\varphi>n)$ is dominated by a polynomially decreasing sequence.
These results are stated in Section~\ref{sec-ex} along with treatments of
exponential decay, stretched exponential decay, and
regularly varying sequences with $\beta=0$.

Next, we mention two theoretical results.
It has been noted elsewhere that either (i) $\varphi\in L^2(Y)$, equivalently
$\sum_{n=1}^\infty n\mu(\varphi>n)<\infty$,
or (ii) summable decay of correlations $\sum_{n=1}^\infty \rho(n)<\infty$,  are sufficient to guarantee the validity of the
central limit theorem.
The special case $q=1$ of Corollary~\ref{cor-poly} below states 
that $\varphi\in L^2(Y)$ implies summable decay of correlations.

\begin{cor} \label{cor-poly}
Let $q>0$.
If $\varphi\in L^{q+1}(Y)$, then
$\sum_{n=1}^\infty n^{q-1}\rho_{v,w}(n)<\infty$ for all
$v\in \mathcal{B}(X)$, $w\in L^\infty(X)$.
\end{cor}

\begin{proof}
By Theorem~\ref{thm-poly}, we can choose $\delta>0$ so that
\[
\rho_{v,w}(n)\ll \sum_{j>\delta n}\mu(\varphi>j)+n\mu(\varphi>\delta n)+n^{-(q+1)}.
\]
(Throughout, we use `big $O$' and $\ll$ notation interchangeably,
writing $a_N=O(b_N)$ or $a_N\ll b_N$ (as $N\to\infty$) if there is
a constant $C>0$ such that $a_N\le Cb_N$ for all $N\ge1$.)

Multiplying by $n^{q-1}$, the last term is summable and the middle term yields
$\sum_{n=1}^\infty n^q\mu(\varphi>\delta n)\ll 
\sum_{n=1}^\infty n^q\mu(\varphi>n)<\infty$.    Finally,
\begin{align*}
\sum_{n=1}^\infty n^{q-1}\sum_{j>\delta n}\mu(\varphi>j) & \ll
\sum_{n=1}^\infty n^{q-1}\sum_{j>n}\mu(\varphi>j)
=\sum_{j=2}^\infty \sum_{n<j}n^{q-1}\mu(\varphi>j)  \\ &
\le \sum_{j=2}^\infty j^q\mu(\varphi>j)<\infty,
\end{align*}
so that $\sum_{n=1}^\infty n^{q-1}\rho_{v,w}(n)<\infty$.
\end{proof}

Our main results, including
Theorem~\ref{thm-poly}, give conditions for uniform rates of decay.  
A natural question is to inquire when uniform decay rates exist in the first place.  The following result addresses this issue.

\begin{thm}
\label{thm-slow}
Suppose that $\mu(\varphi>n)=O((n\log n)^{-1})$.
(We continue to assume in addition that $\varphi\in L^1(Y)$.)
Then correlations decay at a uniform rate for $v\in\mathcal{B}(X)$, $w\in 
L^\infty(X)$.
\end{thm}

\subsection{Systems with excellent inducing schemes}

Let $T:X\to X$ be a transformation with induced map
$F=f^\varphi:Y\to Y$ and $F$-invariant ergodic probability measure $\mu$.
Let $R:L^1(Y)\to L^1(Y)$ be the transfer operator for $F$,
so $\int_Y Rv\,w\,d\mu=
\int_Y v\,w\circ F\,d\mu$ for all $v\in L^1(Y)$, $w\in L^\infty(Y)$.
Define $R_n=R1_{\{\varphi=n\}}$, $n\ge1$.

Let $\mathcal{B}(Y)\subset L^1(Y)$ be a Banach space 
with norm $\|\,\|$ satisfying $|v|_1\le \|v\|$ for all
$v\in\mathcal{B}(Y)$, and such that   
constant functions lie in $\mathcal{B}(Y)$.

\begin{itemize} 
\item[(H1)]  The operator
$R_n:\mathcal{B}(Y)\to\mathcal{B}(Y)$ is bounded 
for all $n$, and $\|R_n\|\ll \mu(\varphi>n)$.
\end{itemize}

Let $\D$ and $\bar\D$ denote the open and closed unit disk in $\C$.
Define $R(z)=\sum_{n=1}^\infty R_nz^n$ for $z\in\bar\D$.  Hypothesis (H1) 
guarantees that $z\mapsto R(z)$ is a continuous family of bounded
operators on $\mathcal{B}(Y)$ for $z\in\bar\D$, and the family is analytic
on $\D$.   Note that $R(1)=R$, so in particular
$1$ lies in the spectrum of $R(1)$.

\begin{itemize} 
\item[(H2)] 
\begin{itemize}
\item[(i)] The eigenvalue $1$ is simple and isolated
in the spectrum of $R(1)$.
\item[(ii)] For $z\in\bar\D\setminus\{1\}$, the spectrum of $R(z)$ does
not contain $1$.
\end{itemize}
\end{itemize}

\begin{defn} \label{def-excellent}
An inducing scheme $F=T^\varphi:Y\to Y$ is {\em excellent} if 
hypotheses (H1) and (H2) are satisfied for an appropriate Banach
space $\mathcal{B}(Y)$.
\end{defn}

Let $v:X\to\R$ be an observable.  We say that $\hat v:Y\to\R$ is {\em derived} from $v$
if for every $n\ge1$, there exists $j\in\{1,\dots,n-1\}$ such that
$\hat v(y)=v(T^jy)$ for all $y\in Y$ with $\varphi(y)=n$.
Let $\mathcal{D}_v$ denote the set of observables $\hat v:Y\to\R$ derived from $v$.

\begin{defn}  An observable
$v:X\to\R$ is {\em exchangeable} if $\mathcal{D}_v$ is a
bounded subset of $\mathcal{B}(Y)$.   Set $\|v\|=\sup_{\hat v\in\mathcal{D}_v}\|\hat v\|$.
\end{defn}

The definition of exchangeability formalises the need for control of iterates
$T^jy$ for $j\in\{1,\dots,\varphi(y)-1\}$.

For excellent inducing schemes and exchangeable observables, we obtain almost
identical 
results as those for H\"older observables on systems modelled by
Young towers.  In particular, Theorems~\ref{thm-poly} and~\ref{thm-slow}
and Corollary~\ref{cor-poly} hold in this generality.   
If we assume further that $\mathcal{B}(Y)$ is embedded in $L^\infty(Y)$ 
(rather than in $L^1(Y)$), then all of our conclusions in 
this paper are identical to those for Young towers.

\begin{example}[Young towers]  \label{ex-Young}
The inducing schemes for the nonuniformly expanding maps studied by Young~\cite{Young98,Young99} 
are Gibbs-Markov.  These are excellent inducing schemes since it
is well-known that hypotheses (H1) and (H2) are satisfied for the
Banach space $\mathcal{B}(Y)$ consisting of piecewise H\"older observables
on $Y$.
Moreover, piecewise H\"older observables on $X$ are exchangeable.

Although most of this paper is concerned with nonuniformly expanding maps, the results extend to systems that are nonuniformly hyperbolic in the sense
of Young~\cite{Young98,Young99}.  Details of this extension are given in 
Appendix~\ref{sec-NUH} based on ideas of~\cite{ChazottesGouezel,GouezelPC}.
\end{example}

\begin{example}[AFN maps]
Zweim\"uller~\cite{Zweimuller98}
studied a class of non-Markovian uniformly expanding
interval maps (so-called AFN maps) with finite absolutely continuous invariant
measures.
In
particular,~\cite{Zweimuller98} obtained a spectral decomposition into
 basic ergodic sets and proved that for each basic set there
is a unique absolutely continuous invariant probability measure.
Each basic set is mixing up to a finite cycle, and we suppose that $X$
is a mixing basic set.
There is a first return map $F:Y\to Y$
that is uniformly expanding
with respect to a partition consisting of intervals.
Moreover $F$ has good distortion properties.
It can be shown that $F$ is an excellent inducing scheme with
function space $\BV(Y)$ (observables of bounded variation).

Unfortunately, $\BV(X)$ is not exchangeable.   However, it turns out that
$F$ is also excellent if we enlarge $\mathcal{B}(Y)$ to consist of
piecewise bounded variable observables, and then the corresponding space
$\mathcal{B}(X)$ is exchangeable.   
The details are sketched in Section~\ref{sec-AFN}.
\end{example}
 
\begin{rmk}   Suppose that the inducing scheme is a first return map 
(that is, $\varphi(y)=\inf\{n\ge1:T^ny\in Y\}$).
If $v:X\to\R$
is supported on $Y$ and $1_Yv\in\mathcal{B}(Y)$, then $v$ is exchangeable.
Hence our results apply to such observables (and all $w\in L^\infty(X)$)
whenever the first return map is an excellent inducing scheme.
\end{rmk}

\subsection{Systems with good inducing schemes}

There are a number of situations where the induced map has good behaviour
but properties such as bounded distortion and/or large images fail.
Examples include the class of interval maps studied by
Ara\'ujo~{\em et al.}~\cite{AraujoLuzzattoViana09} (where the induced map is
of the type studied by Rychlik~\cite{Rychlik83}), and Hu-Vaienti maps~\cite{HuVaienti09} which are multidimensional
nonMarkovian nonuniformly expanding maps with indifferent fixed points.  

In such situations, it is likely that hypothesis (H1) can fail
quite badly.  However, it
turns out that we can obtain decay estimates (often optimal estimates)
under a weaker condition (hypothesis (*) below)
 that seems much more tractable.   Verification of hypothesis (*) in 
situations such as~\cite{AraujoLuzzattoViana09} and~\cite{HuVaienti09} will be 
addressed in future work.

Fix the Banach space $\mathcal{B}(Y)$ as before.
We replace hypothesis (H1) by:
\begin{itemize}
\item[(*)]
$\sum_{n=1}^\infty\sum_{j>n}\|R_j\|<\infty$.
\end{itemize}
Note that $\mu(\varphi=n)=|R_n1_Y|_1\le \|R_n\|\|1_Y\|$ so
for excellent inducing schemes  
condition~(*) is simply the requirement that $\varphi\in L^1(Y)$.
In general, condition~(*) is sufficient to ensure that
the family $R(z)=\sum_{n=1}^\infty R_nz^n$
is analytic on $\D$ and continuous on $\bar\D$ as before.
(The full strength of (*) is required in Proposition~\ref{prop-RR}.)

\begin{defn} \label{def-good}
An inducing scheme $F=T^\varphi:Y\to Y$ 
is {\em good} if hypotheses (*) and (H2) are 
satisfied for an appropriate Banach space $\mathcal{B}(Y)$.
\end{defn}

We have the following generalisations of Theorem~\ref{thm-poly}
and Corollary~\ref{cor-poly}.

\begin{thm} \label{thm-good}
Suppose that $F=T^\varphi:Y\to Y$ is a good inducing scheme and that $\mathcal{B}(X)$ is a collection of exchangeable observables.
Suppose further that $\sum_{n=1}^\infty n^\epsilon\sum_{j>n}\|R_j\|<\infty$ for some $\epsilon>0$.

Then for any $p>0$ there exists $\delta>0$, $C>0$ such that
\[
\rho_{v,w}(n)\le C\|v\| |w|_\infty\Bigl\{\sum_{j>\delta n}\mu(\varphi>j)+n\mu(\varphi>\delta n)+O(n^{-p})\Bigr\},
\]
for all $v\in\mathcal{B}(X)$, $w\in L^\infty(X)$, $n\ge1$.
\end{thm}

\begin{cor} \label{cor-good}
Suppose that $F=T^\varphi:Y\to Y$ is a good inducing scheme and that $\mathcal{B}(X)$ is a collection of exchangeable observables.
Suppose further that $\sum_{n=1}^\infty n^\epsilon\sum_{j>n}\|R_j\|<\infty$ for some $\epsilon>0$.
Let $q>0$.
If $\varphi\in L^{q+1}(Y)$, then
$\sum_{n=1}^\infty n^{q-1}\rho_{v,w}(n)<\infty$ for all
$v\in \mathcal{B}(X)$, $w\in L^\infty(X)$.   \qed
\end{cor}

The remainder of this paper is as follows.
In Section~\ref{sec-strategy}, we describe the strategy adopted in this paper.
In essence, everything that follows Section~\ref{sec-strategy} is an extended exercise.
The required estimates are carried out in 
Sections~\ref{sec-T} and~\ref{sec-L}.
In particular, Section~\ref{sec-L} contains the most general versions of our 
results.
In Section~\ref{sec-ex}, we verify that Theorem~\ref{thm-good} and
Theorem~\ref{thm-slow} follow from the general results and compute 
correlation decay
rates for specific tail functions $\mu(\varphi>n)$.

\begin{rmk}
The technique introduced in this paper can also be used to obtain
a simplified and generalised
treatment of lower bounds (and improved upper bounds) for
decay of correlations~\cite{Gouezel04a,GouezelPhD,Sarig02}.
The results on lower bounds are restricted to the setting 
of excellent first return maps and observables supported on $Y$.   
Since the setting is more restricted, and additional ideas are required, we 
defer these results to a later paper.
\end{rmk}

\section{Strategy}
\label{sec-strategy}

The strategy in this paper consists of three main steps:
\begin{itemize}
\item[1.]   Pass to a 
{\em tower extension} $f:\Delta\to\Delta$ of the
underlying map $T:X\to X$.  The tower $\Delta$ is
a discrete suspension over $F:Y\to Y$ with height $\varphi$.
In particular $F=T^\varphi=f^\varphi$.
Decay of correlations on $\Delta$ pushes down to decay of correlations
on $X$.   Hence this step reduces to the situation where $\varphi$ is
a {\em first} return time function.
\item[2.]   Use {\em dynamical truncation}~\cite{M09}
to replace the tower $\Delta$ by a tower $\Delta'$ with finite height 
$\varphi'$ in such a way that the first return map $F$ is unchanged.
The truncation error between correlation decay on $\Delta$ and on
$\Delta'$ is easily controlled.
\item[3.]   Use {\em operator renewal sequences}~\cite{Gouezel04a,GouezelPhD,Sarig02} to estimate
correlation decay on the truncated tower $\Delta'$ in terms of
the height $\varphi'$ and spectral properties of the transfer operator $R$
for the induced map $F:Y\to Y$.   A key observation from~\cite{M09}
is that the dependence of the estimates on $\varphi'$ are explicit,
while $F:Y\to Y$, $R$ and $\mu$ are unchanged throughout.
\end{itemize}

We now describe each of these steps in more detail.

\subsection{Tower extension}
\label{sec-tower}

Given the induced map $F:Y\to Y$ and return time $\varphi:Y\to\Z^+$,
we define the tower $\Delta=Y^\varphi=\{(y,\ell)\in Y\times \Z:0\le \ell \le \varphi(y)-1\}$ and the tower map $f:\Delta\to\Delta$ by
$f(y,\ell)=(y,\ell+1)$ for $\ell\le\varphi(y)-2$ and
$f(y,\varphi(y)-1)=(Fy,0)$.
Set $\bar\varphi=\int_Y\varphi\,d\mu$ and
define the $f$-invariant probability measure
$\mu_\Delta=(\mu\times{\rm counting})/\bar\varphi$ on $\Delta$.

Define the semiconjugacy $\pi:\Delta\to X$, $\pi(y,\ell)=T^\ell y$,
and set $\mu_X=\pi_*\mu_\Delta$.
Given observables $v,w:X\to\R$, we define the lifted observables
$v\circ\pi,w\circ\pi:\Delta\to\R$.  Then it suffices
to compute  correlation decay rates for the lifted observables on $\Delta$.
Moreover, it is immediate that if $v:X\to\R$ is exchangeable
(relative to the Banach space $\mathcal{B}(Y)$) then so is the lifted
observable $v\circ\pi:\Delta\to\R$.

From now on, given $v,w:\Delta\to\R$, we study decay rates for
\[
\rho_{v,w}(n)=\int_{\Delta} v\,w\circ f^n\,d\mu_{\Delta}-
\int_{\Delta} v\,d\mu_{\Delta} \int_{\Delta} w\,d\mu_{\Delta}.
\]

\subsection{Dynamical truncation}

Given $k\ge1$, we define the truncated return time function
$\varphi'=\min\{\varphi,k\}$.
Just as we defined $f:\Delta\to\Delta$ starting from $F:Y\to Y$ and $\varphi:Y\to\Z^+$,
we can define the truncated tower map $f':\Delta'\to\Delta'$
starting from $F:Y\to Y$ and $\varphi':Y\to\Z^+$.    Note that
$F=f^\varphi=(f')^{\varphi'}$ is independent of $k$.

Similarly, set $\bar\varphi'=\int_Y\varphi'\,d\mu$ and
define the $f'$-invariant probability measure
$\mu_{\Delta'}=(\mu\times{\rm counting})/\bar\varphi'$ on $\Delta'$.

Given $v\in L^\infty(\Delta)$, $w\in L^\infty(\Delta)$, we define
$v\in L^\infty(\Delta')$, $w\in L^\infty(\Delta')$ by restriction.
Let
\[
\rho'_{v,w}(n)=\int_{\Delta'} v\,w\circ (f')^n\,d\mu_{\Delta'}-
\int_{\Delta'} v\,d\mu_{\Delta'} \int_{\Delta'} w\,d\mu_{\Delta'}.
\]
We have the estimate~\cite{M09},
\begin{align} \label{eq-trunc}
|\rho_{v,w}(n)-\rho'_{v,w}(n)|\ll 
|v|_\infty |w|_\infty \Bigl(\sum_{j>k}\mu(\varphi>j)+n\mu(\varphi>k)\Bigr).
\end{align}
See the appendix for details.

\subsection{Operator renewal sequences}
It remains to estimate decay of correlations on the truncated tower.
Since $\varphi'$ is bounded, we expect to obtain an
exponential estimate of the form $|\rho'_{v,w}(n)|\le C_{v,w}(k)e^{-a(k)n}$.
Given sufficient control of $C(k)$ and $a(k)$, this estimate can be
combined with~\eqref{eq-trunc} (choosing $k=k(n)$)
to obtain an estimate for $\rho_{v,w}(n)$.
A surprising aspect of our approach is the degree of control on $C(k)$ and $a(k)$.

We recall the standard definitions of renewal theory, first for
the nontruncated map.
Let $L$ denote the transfer operator for $f:\Delta\to\Delta$
and let $R$ denote the transfer operator for $F=f^\varphi :Y\to Y$.
Define the renewal operators $T_n,R_n:\mathcal{B}(Y)\to\mathcal{B}(Y)$
\[
T_n=1_YL^n1_Y, \enspace n\ge0,\quad
R_n=1_YL^n1_{\{\varphi=n\}}=R1_{\{\varphi=n\}}, \enspace n\ge1.
\]
Define $T(z)=\sum_{n=0}^\infty T_nz^n$ and
$R(z)=\sum_{n=1}^\infty R_nz^n$.    
An elementary calculation shows that
$T_n=\sum_{j=1}^n T_{n-j}R_j$ and hence $T(z)=I+T(z)R(z)$ leading to the renewal equation
$T(z)=(I-R(z))^{-1}$.
Hypothesis (H1) or (*) guarantees that $R(z)$ is analytic on $\D$ and continuous on $\bar\D$.   Moreover, $T(z)$ is analytic on $\D$ and
It follows from (H2)(ii) that $T(z)$ extends continuously to $\bar\D\setminus\{1\}$.  
By (H2)(i), $T(z)$ has a singularity at $z=1$.   The idea of renewal sequences is to use knowledge about
the sequence $R_n$ and the singularity to understand the behaviour of $T(z)$ 
and thereby $T_n$ (and ultimately $L^n$).

The situation is simpler for the truncated dynamical system.
Passing to the truncated tower, we have the transfer operator $L'$ corresponding to 
$f':\Delta'\to\Delta'$.  By construction the first return map $F=(f')^{\varphi'}:Y\to Y$
is independent of $k$ with fixed transfer operator $R$.
Define the truncated renewal operators
\[
T'_n=1_YL'^n1_Y, \enspace n\ge0,\quad
R'_n=1_YL'^n1_{\{\varphi'=n\}}=R1_{\{\varphi'=n\}}, \enspace n\ge1.
\]
Again, $T'(z)=\sum_{n=0}^\infty T'_nz^n$ is analytic on $\D$.
Evidently, $R'_n=0$ for $n>k$, so $R'(z)=\sum_{n=1}^k R'_nz^n$ is a polynomial.
Again, we have the renewal equation $T'(z)=(I-R'(z))^{-1}$.
For the truncated tower, it follows from standard arguments that the singularity of $T'(z)$ at $z=1$ is a simple pole.

In Section~\ref{sec-T}, we investigate the behaviour of $T'(z)$ using the ideas described above.
In Section~\ref{sec-L}, we show how to pass from $T'(z)$ to 
$L'(z)=\sum_{n=0}^\infty L'^nz^n$.   From this we obtain exponential convergence
results for the coefficients $L'^n$ and hence the required
exponential decay
for $\rho'_{v,w}(n)$.

\section{Analyticity of $T'(z)$}
\label{sec-T}

In this section, we assume that we have a good inducing scheme $F:Y\to Y$
with transfer operator $R$ satisfying conditions (*) and (H2) for 
an appropriate Banach space $\mathcal{B}(Y)\subset L^1(Y)$.
Denote the spectral projection corresponding to the simple eigenvalue $1$ for $R(1)$ by
$Pv=\int_Y v\,d\mu$.

For $a\ge0$, let $\D_a=\{|z|\in\C:|z|<e^a\}$. 
Define
\begin{align*}
& S_q(k,a)=\sum_{j=1}^k (\sum_{\ell>j}\|R_\ell\|)j^qe^{ja},
\quad k\ge1,\enspace a,q\in[0,\infty).
\end{align*}
We prove the following result.

\begin{lemma} \label{lem-T}
Let $a=a(k)\in(0,\infty)$ be such that $\lim_{k\to\infty}a^rS_r(k,a)=0$ for some $r\in(0,1]$.
Then there exists $k_0\ge1$ such that for any $q\in(0,1]$, $k\ge k_0$, 
\[
T'(z)=(1-z)^{-1}(1/\bar\varphi')P+J'(z),
\]
where $J'(z)$ is analytic on 
the disk $\D_a$ and $\sup_{z\in \D_a}|z-1|^{1-q}\|J'(z)\|\ll S_q(k,a)$.
\end{lemma}

In the remainder of this section, we prove Lemma~\ref{lem-T}.
As already mentioned, it follows from standard arguments that $T'(z)$ has a simple pole at $z=1$ 
and so $B'(z)=(1-z)T'(z)$ extends analytically to $\D_a$ for some $a>0$.
The proof of Lemma~\ref{lem-T}  consists of estimating $a=a(k)$ and controlling the norms of various analytic families of operators on $\D_a$.
This is a fairly routine exercise, but the calculations are quite complicated.   To remedy this, we first
sketch the formal calculation in Subsection~\ref{sec-formal} and then carry out the rigorous estimates
in Subsection~\ref{sec-rigour}.

\subsection{Formal calculation on $\bar\D$}
\label{sec-formal}

In this subsection, we regard $k$ as fixed and large, and we argue formally.
Note that $R'(1)=R(1)$ with simple isolated eigenvalue $1$.
Moreover $R'(z)$ is a polynomial, so there exists $\delta>0$ such that
the eigenvalue $1$ for $R'(1)$ extends to an analytic family of eigenvalues 
$\lambda'(z)$ on $B_\delta(1)$ with a corresponding family of
spectral projections $P'(z)$.   Let $Q'(z)=I-P'(z)$.
Then in an obvious notation, we can write
\[
T'(z)=(1-\lambda'(z))^{-1}P'(z)+(I-R'(z))^{-1}Q'(z),
\]
for $z\in\bar\D\cap B_\delta(1)$, $z\neq1$.
A standard calculation (eg.~\cite{ParryPoll90}) shows
that $\lambda'(z)=1+(z-1)\bar\varphi'+O(|z-1|^2)$ and hence $T'(z)$
has a pole of order $1$ at $z=1$.  In particular,
the function $B'(z)=(1-z)T'(z)$ is analytic on $\D_a$ for some $a>0$.
Moreover, $B'(1)=(1/\bar\varphi')P$.   Thus we can write
$B'(z)=(1/\bar\varphi')P+(1-z)J'(z)$ where $J'(z)$ is analytic on $\D_a$.
Now divide by $(1-z)$ to obtain the formula for $T'(z)$ in Lemma~\ref{lem-T}.

We end this subsection by deriving a formula for $J'(z)$.   Write
\[
P'(z)=P+(z-1)P_1'(z), \qquad
\lambda'(z)=1+(z-1)\{\bar\varphi'+\tilde\lambda'(z)\},
\]
where $\tilde\lambda'(1)=0$.  Then (at least formally),
\[
\Bigl(\frac{1-\lambda'(z)}{1-z}\Bigr)^{-1}
= \frac{1}{\bar\varphi'}\Bigl(1+\frac{1}{\bar\varphi'}\tilde\lambda'(z)\Bigr)^{-1}= 
\frac{1}{\bar\varphi'}-\Bigl(\frac{1}{\bar\varphi'}\Bigr)^2
\tilde\lambda'(z)\Bigl(1+\frac{1}{\bar\varphi'}\tilde\lambda'(z)\Bigr)^{-1},
\]
and hence
\begin{align}  \label{eq-H1}
J'(z)=\begin{cases} (z-1)^{-1}(1/\bar\varphi')^2\tilde\lambda'(z)\bigl\{1+(1/\bar\varphi')\tilde\lambda'(z)\bigr\}^{-1}P \\[.75ex]
 \qquad \qquad  \qquad
-(1/\bar\varphi')\{1+(1/\bar\varphi')\tilde\lambda'(z)\bigr\}^{-1}P_1'(z)  \\[.75ex]
 \qquad \qquad  \qquad
+(I-R'(z))^{-1}Q'(z), & z\in \D_a\cap B_\delta(1) \\[1.25ex]
(I-R'(z))^{-1}\;-\;(1-z)^{-1}(1/\bar\varphi')P, & z\in\D_a\setminus B_\delta(1)
\end{cases}
\end{align}

\subsection{Rigorous calculation on $\D_a$}
\label{sec-rigour}

By (H2)(i), we can choose a closed loop $\Gamma\in\C\setminus\spec R(1)$ 
separating $1$ from the 
remainder of the spectrum of $R(1)$.
There exists $\delta>0$ such that the spectrum of
$R(z)$ does not intersect $\Gamma$ for $z\in \bar\D\cap B_\delta(1)$ and we can define
the spectral projection
\begin{align} \label{eq-P}
P(z)=\frac{1}{2\pi i}\int_\Gamma (\xi I-R(z))^{-1}\,d\xi.
\end{align}
For $z\in \bar\D\cap B_\delta(1)$, define the corresponding eigenvalue $\lambda(z)$, so $R(z)P(z)=\lambda(z)P(z)$,
and the complementary projection $Q(z)=I-P(z)$.

For $k$ sufficiently large, and $z$ close enough to $1$,
we can define similarly $\lambda'(z)$, $P'(z)$ and $Q'(z)$.
The next result is a uniform version of this statement.

\begin{prop}  \label{prop-first}
Suppose that $a=a(k)$ satisfies $\lim_{k\to\infty}aS_0(k,a)\to0$. Then
\begin{itemize}
\item[(a)]  For any $\delta>0$, there exists $k_0\ge1$
such that $\|(I-R'(z))^{-1}\|\ll1$ for $k\ge k_0$,
$z\in\D_a\setminus B_\delta(1)$.
\item[(b)] There exists $\delta>0$ and $k_0\ge1$ such that
for all $k\ge k_0$ there exists a continuous family 
$z\mapsto\lambda'(z)$, $z\in\D_a\cap B_\delta(1)$, of
simple eigenvalues for $R'(z)$ satisfying $\lambda'(1)=1$.
Moreover, $|\lambda'(z)|\ll 1$ for $k\ge k_0$,
$z\in\D_a\cap B_\delta(1)$.
\item[(c)] The spectral projections $P'(z)$ corresponding to the
eigenvalues $\lambda'(z)$ satisfy $\|P'(z)\|\ll 1$ for $k\ge k_0$,
$z\in\D_a\cap B_\delta(1)$.
\item[(d)]  $\|(I-R'(z))^{-1}Q'(z)\|\ll 1$ for $k\ge k_0$,
$z\in\D_a\cap B_\delta(1)$.
\end{itemize}
\end{prop}

\begin{proof}  
We break the proof into three steps.   First we work with $R(z)$, etc, 
on $\bar\D$.   Second, we consider $R'(z)$, etc, on $\bar\D$.
Third, we consider $R'(z)$, etc, on $\D_a$.

\noindent 1.)
By (*), $R(z)$ is uniformly convergent and hence continuous on $\bar\D$.
Thus the results for $(I-R(z))^{-1}$, $\lambda(z)$, $P(z)$
and $(I-R(z))^{-1}Q(z)$ follow from (H2).

\noindent 2.)
Note that $\|R(z)-R'(z)\|\le 2\sum_{j>k}\|R_j\|\to0$ as $k\to\infty$
uniformly on $\bar\D$ by (*).  Hence the results for $(I-R'(z))^{-1}$, 
$\lambda'(z)$, $P'(z)$ and $(I-R'(z))^{-1}Q'(z)$ on $\bar\D$ follow from step~1
and the resolvent identity.

\noindent 3.)
We claim that $\|R'(e^{a+ib})-R'(e^a)\|\ll a S_0(k,a)$.  By
assumption, $aS_0(k,a)\to0$
as $k\to0$, so the result follows from step~2
and the resolvent identity.

To verify the claim, compute that
\begin{align*}
& \|R'(e^{a+ib})-R'(e^{ib})\|  \le \sum_{j=1}^k\|R'_j\|(e^{ja}-1)
\\ & = \sum_{j=1}^k (\sum_{\ell\ge j}\|R'_\ell\|)(e^{ja}-1)
- \sum_{j=1}^k (\sum_{\ell\ge j+1}\|R'_\ell\|)(e^{ja}-1)
\\ & = \sum_{j=1}^k (\sum_{\ell\ge j}\|R'_\ell\|)(e^{ja}-1)
- \sum_{j=1}^k (\sum_{\ell\ge j}\|R'_\ell\|)(e^{(j-1)a}-1)
\\ & = (e^a-1)\sum_{j=0}^{k-1} (\sum_{\ell >j}\|R'_\ell\|)e^{ja}
= (e^a-1)\sum_{j=0}^{k-1} (\sum_{\ell >j}\|R_\ell\|)e^{ja}
= (e^a-1)S_0(k,a),
\end{align*}
as required.
\end{proof}

Define the polynomials of degree $k-1$,
\[
R_1'(z)=\frac{R'(z)-R'(1)}{z-1}, \quad
\tilde R'(z)=\frac{R'(z)-R'(1)}{z-1}-\frac{dR'}{dz}(1)=R'_1(z)-R'_1(1).
\]
Similarly, starting from $P'$ and $\lambda'$ instead of $R'$, define the analytic functions
$P_1'(z)$, $\tilde P'(z)$, $\lambda_1'(z)$ and $\tilde \lambda'(z)$.

\begin{prop}    \label{prop-tilde}
Suppose that $a=a(k)$ satisfies $\lim_{k\to\infty}aS_0(k,a)\to0$. 
There exists $\delta>0$ and $k_0\ge1$ such that for all $k\ge k_0$
and $z\in \D_a\cap B_\delta(1)$,
\begin{itemize}
\item[(a)] $\|P'_1(z)\|\ll \|R'_1(z)\|$ and $|\lambda'_1(z)| \ll \|R'_1(z)\|$.
\item[(b)] $|\tilde\lambda'(z)|\ll \|\tilde R'(z)\|+|z-1|\|R_1'(z)\|^2$.
\end{itemize}
\end{prop}

\begin{proof}
(a) The estimate for $P'_1$ follows from 
equation~\eqref{eq-P} and the resolvent identity.
Next,
\begin{align*}
(\lambda'(z)-1)P & = 
\lambda'(z)P'(z)-\lambda'(1)P'(1)\;-\; \lambda'(z)(P'(z)-P'(1)) \\
& = R'(z)P'(z)-R'(1)P'(1)\;-\; \lambda'(z)(P'(z)-P'(1)),
\end{align*}
so the estimate for $\lambda'_1$ follows from the estimate for $P'_1$.

\noindent(b)   By equation~\eqref{eq-P}, $\tilde P'(z)=\frac{1}{2\pi i}\int_\Gamma (I+II)\,d\xi$, where 
\begin{align*}
I & =(\xi I-R'(1))^{-1}\tilde R'(z)(\xi I-R'(1))^{-1}   \\[.75ex]
II & = \{(\xi I-R'(z))^{-1}-(\xi I-R'(1))^{-1}\}
R_1'(z)(\xi I-R'(1))^{-1}.
\end{align*}
Then $\|I\|\ll \|\tilde R'(z)\|$ and 
$\|II\|\ll |z-1|\|R_1'(z)\|^2$.   
Hence $\|\tilde P'(z)\|\ll \|\tilde R'(z)\|+{|z-1|}\|R_1'(z)\|^2$.

Next, define $S'(z)=R'(z)P'(z)$ and correspondingly $\tilde S'(z)$.
Then
\[
\tilde S'(z)=R'(z)\tilde P'(z)+\tilde R'(z)P'(1)+(R'(z)-R'(1))\frac{dP'}{dz}(1).
\]
so that $\|\tilde S'(z)\|\ll \|\tilde R'(z)\|+|z-1|\|R_1'(z)\|^2$.
Finally, 
\[
\tilde\lambda'(z)P=\tilde S'(z)-\tilde P'(z)-(z-1)^{-1}(\lambda'(z)-\lambda'(1))(P'(z)-P'(1)),
\]
yielding the required estimate for $\tilde \lambda'$.
\end{proof}

\begin{prop} \label{prop-RR}   Let $r\in[0,1]$, $a>0$, $b\in[0,2\pi]$.  
\begin{itemize} 
\item[(a)] $\|R'_1(e^{a+ib})-R'_1(e^{ib})\|\ll a^rS_r(k,a)$.
\item[(b)] $\|R'_1(z)\|\ll 1+a^rS_r(k,a)$ for all $z\in \D_a$.
\item[(c)] For any $\epsilon>0$, there exists $\delta>0$, $k_0\ge1$ such that
$\|\tilde R'(z)\|\le \epsilon +a^rS_r(k,a)$ for all 
$k\ge k_0$ and $z\in\D_a\cap B_\delta(1)$.
\end{itemize} 
\end{prop}

\begin{proof}
Write $U_j=\sum_{\ell >j}R_\ell$.
The same calculation as the one used in the proof of 
Proposition~\ref{prop-first} shows that
$R'_1(z)=\sum_{j=1}^k\sum_{\ell\ge j}R'_\ell z^{j-1}=
\sum_{j=0}^{k-1} U_jz^j$.
Hence
\begin{align*}
\|R'_1(e^{a+ib})-R'_1(e^{ib})\| & \le \sum_{j=1}^{k-1} \|U_j\|(e^{ja}-1)
 \le a\sum_{j=1}^M j\|U_j\| e^{ja}+
\sum_{j=M}^k \|U_j\| e^{ja} \\
 & \le aM^{1-r}\sum_{j=1}^M j^r\|U_j\| e^{ja}+ M^{-r}\sum_{j=M}^k j^r \|U_j\| e^{ja} 
\end{align*}
so taking $M\approx 1/a$ yields part (a).

By (*), $R'_1$ is bounded on $\bar\D$ uniformly in $k$.
Hence (b) follows from (a).

Let $S(z)=\sum_{j=0}^\infty U_jz^j$
and note that $S$ is absolutely summable on $\bar\D$ by (*).
In particular, $R'_1$ and $S$ are continuous on $\bar\D$.
Moreover, $S$ is independent of $k$ and we can choose $\delta$ so
that $\|S(z)-S(1)\|<\epsilon/2$ for $z\in\bar\D\cap B_\delta(1)$.   Choose $k_0$ so that $\|S(z)-R'_1(z)\|
<\epsilon/4$ for all $z\in\bar\D$, $k\ge k_0$.
Writing $\tilde R'(z)=R'_1(z)-R'_1(1)=(S(z)-S(1))-(S(z)-R'_1(z))+(S(1)-R'_1(1))$
we obtain that
$\|\tilde R'\|<\epsilon$ for all $z\in\bar\D\cap B_\delta(1)$,
$k\ge k_0$.
Hence (c) follows from (a).
\end{proof}

Recall that the definition of $J'$ in~\eqref{eq-H1} relied on the invertibility of
$1+(1/\bar\varphi')\tilde\lambda'(z)$.

\begin{cor} \label{cor-J}
If $a^rS_r(k,a)\to0$ for some $r\in(0,1]$, then there exists $k_0\ge1$ such that
$J'$ is well-defined on $\D_a$ and
$\|J'(z)\|\ll |z-1|^{-1}\|\tilde R'(z)\|+1$ for all $k\ge k_0$
and $z\in\D_a$.
\end{cor}

\begin{proof}
In particular, $aS_0(k,a)\to0$, so Proposition~\ref{prop-first} applies. Hence
$P'$, $\lambda'$, and so on exist and are uniformly bounded on $\D_a$.
By Propositions~\ref{prop-tilde} and~\ref{prop-RR}, $|\tilde\lambda'|\ll
\epsilon+a^rS_r(k,a)+(z-1)(1+aS_r(k,a))^2$, so
choosing $\delta$ and $k_0$ appropriately, we can arrange that
$(1/\bar\varphi')|\tilde\lambda'|<\frac12$ (say).  Hence
the formal expression for $J'$ makes sense.   Moreover all terms in this expression
are uniformly bounded except possibly for $P_1'$ and $(z-1)^{-1}\tilde\lambda'$.
By  Propositions~\ref{prop-tilde} and~\ref{prop-RR}, $\|P'_1\|\ll \|R'_1\|\ll 1+a^rS_r(k,a)\ll1$,
and $|(z-1)^{-1}\tilde\lambda'|\ll |z-1|^{-1}\|\tilde R'\|+\|R'_1\|^2 \ll  |z-1|^{-1}\|\tilde R'\|+1$.
\end{proof}

\begin{pfof}{Lemma~\ref{lem-T}}
Write $U_j=\sum_{\ell >j}R_\ell$.  Then
$\tilde R'(z) =\sum_{j=0}^{k-1}U_j(z^j-1)$ and so
\begin{align*}
\|\tilde R'(z)\| & 
 \ll 
|z-1|\sum_{j=0}^M j \|U_j\| e^{ja}+\sum_{j=M}^k \|U_j\| e^{ja} \\
& \ll |z-1|M^{1-q}\sum_{j=0}^M j^q \|U_j\| 
+M^{-q}\sum_{j=M}^k j^q \|U_j\|.
\end{align*}
Now take $M\approx 1/|z-1|$ to deduce that $\|\tilde R'\|\ll |z-1|^qS_q(k,a)$.
Hence the result follows from Corollary~\ref{cor-J}.
\end{pfof}

\section{Analyticity of $L'(z)$ and the main result}
\label{sec-L}

In this section, we show how to pass from 
$T'(z)=\sum_{n=0}^\infty 1_YL'^n1_Yz^n:\mathcal{B}(Y)\to\mathcal{B}(Y)$ to
$L'(z)=\sum_{n=0}^\infty L'^nz^n:\mathcal{B}(\Delta')\to L^1(\Delta')$.
We continue to suppose that $F$ is a good inducing scheme.
(Recall that $\mathcal{B}(\Delta')$ is the collection of exchangeable
observables.)

If in addition $\mathcal{B}(Y)$ is embedded in $L^\infty(Y)$, then our
results are identical to those for $T'(z)$ while in general we have to be content with cruder estimates that are still sufficient for the results mentioned
in the introduction.

Let $P_{\Delta'}$ denote the 
projection $P_{\Delta'}v=\int_{\Delta'} v\,d\mu_{\Delta'}$.

\begin{lemma} \label{lem-L}
(i) Suppose that $F$ is a good inducing scheme and 
in addition that $\mathcal{B}(Y)$ is embedded in $L^\infty(Y)$.
Let $a=a(k)$ be such that $\lim_{k\to\infty}a^rS_r(k,a)=0$ for some $r\in(0,1]$.
Then there exists $k_0\ge1$ such that for any $q\in(0,1]$, 
$k\ge k_0$, 
\begin{align} \label{eq-L}
L'(z)=(1-z)^{-1}P_{\Delta'}+H'(z)+E'(z),
\end{align}
where $E'(z)$ is a polynomial
of degree at most $k-1$, $H'(z)$ is analytic on the 
disk $\D_a$ and $\sup_{z\in\D_a}|z-1|^{1-q}\|H'(z)\|\ll S_q(k,a)$.

(ii) In the general case of good inducing schemes, the same result holds
except that $\sup_{z\in\D_a}\|H'(z)\|\ll k^2e^{2ka}$.
\end{lemma}

We can now state and prove our main result.

\begin{thm} \label{thm-main}
(i) Suppose that $F$ is a good inducing scheme and 
in addition that $\mathcal{B}(Y)$ is embedded in $L^\infty(Y)$.
 Let $a=a(k)$ be such that $\lim_{k\to\infty}a^rS_r(k,a)=0$ for some $r\in(0,1]$.
Let $q\in(0,1]$.
 Then there exists $C>0$, $k_0\ge1$ such that 
 \[
 |\rho_{v,w}(n)|\le 
 C|v|_\infty |w|_\infty \Bigl(\sum_{j>k}\mu(\varphi>j)+n\mu(\varphi>k)\Bigr)
 +C\|v\||w|_\infty S_q(k,a)e^{-na},
 \]
 for all $v\in \mathcal{B}(X)$, $w\in L^\infty(X)$, $n\ge k\ge k_0$.

(ii) In the general case of good inducing schemes, the same result holds but
with $S_q(k,a)$ replaced by $k^2e^{2ka}$.
 \end{thm}

\begin{proof}
Suppose that we are in case (i).
Write $H'(z)=\sum_{n=0}^\infty H'_n z^n$, $E'(z)=\sum_{j=0}^{k-1}E_n' z^n$.
Equating coefficients in~\eqref{eq-L} on the open unit disk $\D$, we obtain
\[
L'^n=P_{\Delta'}+H_n'+E'_n,
\]
for $n\ge1$, $k\ge k_0$.    We claim that $\|H'_n\|\ll  S_q(k,a)e^{-na}$
for all  $n\ge1$, $k\ge k_0$.  
It follows that $\|L'_n-P_{\Delta'}\|=\|H'_n\|\ll  S_q(k,a)e^{-na}$ for all $n\ge k\ge k_0$.  Hence
$|\rho'_{v,w}(n)|\ll \|v\||w|_1  S_q(k,a)e^{-na}$ for all $v\in\mathcal{B}(\Delta)$, $w\in L^1(\Delta)$.
The result follows from this estimate combined with~\eqref{eq-trunc}.

It remains to prove the claim.   Since $H'$ is analytic on $\D_a$,
$\|H'_n\|\ll \int_\Gamma \|H'(z)z^{-n}\|dz$ where
$\Gamma$ is the boundary circle of $\D_a$ (for a slightly smaller $a$).   Hence
\begin{align*}
\|H'_n\|\ll e^{-na}\int_0^{2\pi}\|H'(e^{a+ib})\|\,db \ll S_q(k,a)e^{-na}  \int_0^{2\pi}|e^{a+ib}-1|^{-(1-q)}\,db.
\end{align*}
But $|e^{a+ib}-1|\ge |e^a\sin b|\ge|\sin b|$, so
\begin{align*}
 \int_0^{2\pi}|e^{a+ib}-1|^{-(1-q)}\,db & \le  \int_0^{2\pi}|\sin b|^{-(1-q)}\,db=4\int_0^{\pi/2}|\sin b|^{-(1-q)}\,db
\\ & \ll \int_0^{\pi/2}|b|^{-(1-q)}\,db\ll1,
\end{align*}
completing the proof of the claim and hence of case (i).  The proof of case (ii)
is similar.
\end{proof}

\begin{rmk}\label{rmk-main}   
The statement of Theorem~\ref{thm-main}(i) is sufficiently
general for all of our applications except in Example~\ref{ex-sv}
where it is necessary to improve the factor $S_q(k,a)$.    
Such improvements can be achieved by modifying the estimate of $\tilde R(z)$
obtained at the end of the proof of Lemma~\ref{lem-T}.
\end{rmk}

In the remainder of this section, we prove Lemma~\ref{lem-L}.
We focus on case (i), sketching the differences
for case (ii) at the end of the proof.

Write $(L'^nv)(x)=\sum_{f'^nu=x}g'_n(u)v(u)$.
Define operator-valued polynomials 
\[
A'(z):L^\infty(Y)\to L^1(\Delta), \quad
D'(z):\mathcal{B}(\Delta')\to \mathcal{B}(Y), \quad
E'(z):L^\infty(\Delta')\to L^1(\Delta').
\]
as follows:
\begin{align*}
A'(z)& =\sum_{n=0}^{k-1} A'_nz^n, \quad
  (A'_nv)(x) =\sum_{\substack{ f'^ny=x \\ y\in Y;\;f'y\not\in Y,\ldots,f'^ny\not\in Y }}\!\!\!\!\!\!\!\!  g'_n(y)v(y), \\
D'(z)& =\sum_{n=0}^{k-1} D'_nz^n, \quad
 (D'_nv)(y) =\sum_{\substack{ f'^nu=y \\ u\not\in Y,\ldots,f'^{n-1}u\not\in Y;\;
f'^nu\in Y }}\!\!\!\!\!\!\!\! g'_n(u)v(u), \\
E'(z)& =\sum_{n=1}^{k-1} E'_nz^n, \quad
 (E'_nv)(x)  =\sum_{\substack{ f'^nu=x \\ u\not\in Y,\ldots,f'^nu\not\in Y }}
\!\!\!\!\!\!\!\! g'_n(u)v(u).
\end{align*}
(We adopt the convention that $(A'_0v)(y,\ell)=v(y)$ for $\ell=0$ and is zero otherwise, and that $(D'_0v)(y)=v(y,0)$.)

Following Gou\"ezel~\cite{GouezelPhD}, we observe that
\[
\SMALL {L'}^n=\sum_{n_1+n_2+n_3=n}A'_{n_1}T'_{n_2}D'_{n_3}+E'_n.
\]
Hence $L'(z)=A'(z)T'(z)D'(z)+E'(z) :\mathcal{B}(\Delta')\to L^1(\Delta')$.

\begin{prop}  \label{prop-AD}
Let $k\ge1$, $0\le n \le k-1$.   Then
\begin{itemize}
\item[(a)] For $v\in L^1(\Delta')$,
$(A'_nv)(y,\ell)=v(y)$ if $n=\ell$ and is zero otherwise.

\item[(b)]  $(A'(1)v)(y,\ell)=v(y)$ for all $(y,\ell)\in\Delta'$.

\item[(c)] For all $v\in\mathcal{B}(\Delta)$, 
$D'_nv= R(1_{\{\varphi>n\}}\hat v_{k,n})=\sum_{j>n}R_j\hat v_{k,n}$,
where $\hat v_{k,n}\in\mathcal{B}(Y)$ is derived from $v$
(so in particular, $\|\hat v_{k,n}\|\le \|v\|$ for all $k,n$).

\item[(d)] $D'(1)v=RV'$
where $V'(y)=\sum_{\ell=0}^{\varphi'(y)-1}v(y,\ell)$.
\end{itemize}
\end{prop}

\begin{proof}
Parts (a) and (b) are immediate from the definitions.

Write $(Rv)(y)=\sum_{Fu=y}G(u)v(u)$.
Then
$(D'_nv)(y)={\sum}^* G(u)v(u,\varphi'(u)-n)$
where the summation is over $u\in Y$ with $Fu=y$ and $\varphi(u)>n$.
Parts (c) and (d) follow easily.
\end{proof}

Define $A_1'(z)=(1-z)^{-1}(A'(z)-A'(1))$, $\,D_1'(z)=(1-z)^{-1}(D'(z)-D'(1))$.

\begin{cor}  \label{cor-AD}
\begin{itemize}
\item[(a)] $(1/\bar\varphi')A'(1)PD'(1)=P_{\Delta'}$.
\item[(b)]  $\|A'(z)\|_{L^\infty(Y)\to L^1(\Delta)}\ll 1+a^rS_r(k,a)$ and  $\|A_1'(z)\|_{L^\infty(Y)\to L^1(\Delta)}\ll |z-1|^{-(1-q)}S_q(k,a)$ 
for all $q,r\in(0,1]$, $a>0$, $z\in\D_a$.
\item[(c)]  $\|D'(z)\|_{B(\Delta')\to B(Y)}\ll 1+a^rS_r(k,a)$ and  $\|D_1'(z)\|_{B(\Delta')\to B(Y)}\ll |z-1|^{-(1-q)}S_q(k,a)$ 
for all $q,r\in(0,1]$, $a>0$, $z\in\D_a$.
\end{itemize}
\end{cor}

\begin{proof}
By Proposition~\ref{prop-AD}(b,d) and the definition of $\mu_{\Delta'}$,
\begin{align*}
(A'(1)PD'(1)v)(y,\ell) & =(PD'(1)v)(y)=\int_Y RV'\,d\mu=\int_Y V'\,d\mu \\ &
=\int_Y\sum_{\ell=0}^{\varphi'(y)-1}v(y,\ell)\,d\mu
=\bar\varphi'\int_{\Delta'} v\,d\mu_{\Delta'},
\end{align*}
proving part (a).

By Proposition~\ref{prop-AD}(a), the support of $A'_n$ has measure $\mu(\varphi\ge n)$ and 
$|A'_nv|_\infty\le |v|_\infty$.
It follows that $|A'_nv|_1\le \mu(\varphi\ge n)|v|_\infty$.
In other words, $\|A'_n\|\le \mu(\varphi\ge n)$.
Hence the estimates in part (b) are obtained in exactly the same
way as the estimates for $R_1'(z)$ and $\tilde R'(z)$  in 
the proof of Lemma~\ref{lem-T}.

By Proposition~\ref{prop-AD}(c), $\|D'_n\|\ll \sum_{j>n}\|R'_j\|$.
Hence the estimates in part (c) are again obtained in exactly the same
way as the estimates for $R_1'(z)$ and $\tilde R'(z)$.
\end{proof}

\begin{pfof}{Lemma~\ref{lem-L}}
By Lemma~\ref{lem-T} and Corollary~\ref{cor-AD}(a),
\begin{align*}
A'(z)T'(z)D'(z) & =(1-z)^{-1}(1/\bar\varphi')\Bigl\{
A'(1)PD'(1)+(A'(z)-A'(1))PD'(1) \\ 
& \qquad\qquad\qquad\qquad \qquad + A'(z)P(D'(z)-D'(1))\Bigr\} \\
& \qquad + A'(z)(T'(z)-(1-z)^{-1}(1/\bar\varphi')P)D'(z) 
\end{align*}
\begin{align*}
&\qquad = (1-z)^{-1}P_{\Delta'}+(1/\bar\varphi')A_1'(z)PD'(1)+(1/\bar\varphi')A'(z)PD_1'(z)+A'(z)J'(z)D'(z),
\end{align*}
and so
\[
L'(z)=
(1-z)^{-1}P_{\Delta'}+(1/\bar\varphi')A_1'(z)PD'(1)+(1/\bar\varphi')A'(z)PD_1'(z)+A'(z)J'(z)D'(z)+E'(z).
\]
Hence, case (i) follows immediately from 
Lemma~\ref{lem-T} and Corollary~\ref{cor-AD}(b,c).

In case (ii), we replace 
Corollary~\ref{cor-AD}(b) by the crude estimates
$\|A'(z)\|_{L^1(Y)\to L^1(\Delta)}\ll ke^{ka}$ and $\|A_1'(z)\|_{L^1(Y)\to L^1(\Delta)}\ll k^2e^{2ka}$.
\end{pfof}

\section{Examples}
\label{sec-ex}

In this section, we consider a number of special cases of 
Theorem~\ref{thm-main}, including the proofs of
the results stated in the introduction.
In Subsection~\ref{sec-AFN}, we verify that the AFN maps described
in the introduction have the desired properties.

\subsection{Calculations for good inducing schemes}

In this subsection, we describe results that follow from
Theorem~\ref{thm-main}(ii).  It is assumed that $F:Y\to Y$ is
a good inducing scheme with $\mathcal{B}(Y)$ embedded in $L^1(Y)$,
and that $\mathcal{B}(X)$ is exchangeable.   The only control required
on $\|R_n\|$ is that  
$\sum_{n=1}^\infty n^\epsilon\sum_{j>n}\|R_j\|<\infty$ for some $\epsilon>0$.

\begin{pfof}{Theorem~\ref{thm-good}}
Let $a(k)=\frac12 \log k/k$.   For $r<\epsilon$, we compute that
\[
S_r(k,a)\le S_r(k,0)e^{ka}\ll e^{ka}=k^{\frac12},
\]
and so $a^rS_r(k,a)\to0$.  By Theorem~\ref{thm-main}(ii), we obtain the estimate
\[
|\rho(n)|\ll \sum_{j>k}\mu(\varphi>j)+n\mu(\varphi>k)
+O(k^3 e^{-\frac12 n\log k/k}).
\]
Now take $k=\delta n$ with $\delta=1/(2p+6)$.
\end{pfof}

\begin{example}[Polynomially decreasing sequences{~\cite[D\'efinition~2.2.11]{GouezelPhD}}] \label{ex-poly}
Suppose that $\sum_{j>n}\|R_j\|=O(1/n^{1+\epsilon})$ for some $\epsilon>0$, and that $\mu(\varphi>n)\ll
u_n$ where $u_n$ has the property that there exists a constant $C>0$
such that  $u_j\le Cu_n$ for all $n\ge1$, $j\ge n/2$.
Then we obtain the optimal upper bound $\rho(n)\ll\sum_{j>n}u_j$.

To see this, first observe that $u_j\le C^k u_n$ for all $j\ge n/2^k$
and taking $2^k\approx n$ we deduce that $n^{-p}\ll u_n$ for
some $p>0$.    Also, for any $\delta>0$ there exists $C=C(\delta)$
such that $u_{[\delta n]}\le Cu_n$.
It follows that $\sum_{j>\delta n}\mu(\varphi>j)\ll \sum_{j>n}u_j$.

Finally, $\frac{n}{2}\mu(\varphi>n)\le 
\sum_{j\ge n/2}\mu(\varphi>j)\ll \sum_{j> n/2}u_j\ll \sum_{j> n}u_{j/2}
\ll \sum_{j> n}u_j$ so 
$n\mu(\varphi>\delta n)\ll \sum_{j>\delta n}u_j\ll \sum_{j>n}u_j$.
This accounts for all the terms in Theorem~\ref{thm-good}.
\end{example}

\begin{example}[Regularly varying sequences, $\beta>0$]
We continue to assume that 
$\sum_{j>n}\|R_j\|=O(1/n^{1+\epsilon})$ for some $\epsilon>0$.
Suppose further that 
$\mu(\varphi>n)\ll u_n=\ell(n)/n^{\beta+1}$ where $\beta>0$
and $\ell$ is a slowly varying function.
(Recall that $\ell:(0,\infty)\to(0,\infty)$ is {\em slowly varying} if
$\lim_{x\to\infty}\ell(\lambda x)/\ell(x)=1$ for all $\lambda>0$.)

Regularly varying sequences are clearly
polynomially decreasing, so we can apply the result of Example~\ref{ex-poly}.
Moreover, $\sum_{j>n}u_j\ll \ell(n)/n^{\beta}$ by a result
of Karamata (see~\cite[Theorem~1, p.~273]{Feller66}).
Hence we obtain the optimal upper bound $|\rho(n)|\ll\ell(n)/n^{\beta}$.
\end{example}

Finally, we consider the standard case of exponential decay of correlations.
\begin{example}[Exponential decay rates]
If $\|R_n\|=O(e^{-cn})$, $c>0$,
then we obtain exponential decay of correlations as expected.
A slight reformulation of Theorem~\ref{thm-main}(ii) is required, where we
modify the condition $\lim_{k\to\infty}a^rS_r(k,a)=0$.
Indeed the only places where the 
condition is used (rather than simply boundedness)
is in Step 3 of the proof of Proposition~\ref{prop-first} and in ensuring that $(1/\varphi')|\tilde\lambda'|<\frac12$ in the proof of 
Corollary~\ref{cor-J}.   
For these it suffices that for any $\epsilon>0$, there exists 
$a=a(k)$ and $r\in(0,1]$ such that $a^rS_r(k,a)<\epsilon$.   

Under the assumption $\|R_n\|=O(e^{-cn})$,
this new condition can be satisfied with $r=1$ and
$a$ chosen to be a sufficiently
small constant $a\equiv\epsilon_1\in(0,c)$.   
(Taking $\epsilon_1<c$ ensures that $S_1(k,a)$ is bounded; the other
requirements on $\epsilon_1$ are less explicit.)
Let $n=k$.  Since $\mu(\varphi=n)\ll \|R_n\|=O(e^{-cn})$,  we obtain
$|\rho(n)|\ll n^2e^{-\epsilon_1 n}$.
\end{example}

\subsection{Calculations for inducing schemes with $\mathcal{B}(Y)\subset L^\infty(Y)$}

In this section, we suppose that $F:Y\to Y$ is an excellent inducing scheme
and that $\mathcal{B}(Y)$ is embedded in $L^\infty(X)$.   As usual, 
we suppose that $\mathcal{B}(X)$ is exchangeable.

Since $\mathcal{B}(Y)$ is embedded in $L^\infty(X)$, we can appeal to
part (i) of Theorem~\ref{thm-main}.   Since $F$ is excellent, hypotheses on
$\varphi$ are inherited by $\|R_n\|$.   (The results in this subsection can
be formulated for good inducing schemes by imposing conditions on $\|R_n\|$
directly but the ensuing results are suboptimal.)

\begin{pfof}{Theorem~\ref{thm-slow}}
We take $q=r=1$ in Theorem~\ref{thm-main}(i).
Let $a=\frac12 k^{-1}\log\log k$.
Then, $S_1(k,a)\le \sum_{j=1}^k (\log^{-1} j) e^{ak}\ll k\log^{-\frac12}k$,
and so $aS_1(k,a)\to0$.
Moreover, $S_1(k,a)e^{-na}\ll  k(\log^{-\frac12}k)e^{-na}\to0$ with
$n=2k\log k/\log\log k$.   In addition,
$n\mu(\varphi>k)\ll (\log\log k)^{-1}\to0$.
Finally, $\sum_{j>k}\mu(\varphi>j)\to0$ since $\varphi\in L^1(Y)$.~
\end{pfof}

\begin{example}[Regularly varying sequences, $\beta=0$]
\label{ex-sv}
We consider the case of regularly varying sequences
$\mu(\varphi>n)\ll \ell(n)/n$ where $\ell(n)\to0$ as $n\to\infty$
(supposing as always that $\varphi\in L^1(Y)$).   
Many such examples were considered by Holland~\cite{Holland05}.

We suppose also that $\ell(n)$ is decreasing.
It follows that $\ell(n)\log n$ is bounded
(since $\ell(n)\log n\ll \ell(n)\sum_{j=1}^n 1/j\le
\sum_{j=1}^n \ell(j)/j\ll 1$).

Take $a(k)=\frac12 k^{-1}\log(1/\ell(k))$.  By Karamata,
\[
S_1(k,a)\le S_1(k,0)e^{ka}\le \ell(k)^{-\frac12}\sum_{j=1}^k \ell(j)\ll k\ell(k)^\frac12,
\]
and it follows that $\lim_{k\to\infty}aS_1(k,a)=0$.   

As mentioned in Remark~\ref{rmk-main}, we require a refinement to the 
estimate of $\|\tilde R'\|$ at the end of the proof of Lemma~\ref{lem-T}.
Recall that $\tilde R'(z) =\sum_{j=0}^{k-1}U_j(z^j-1)$ 
where $U_j=\sum_{\ell >j}R_\ell$.   By assumption,
$\|U_j\|\ll \ell(j)/j$.
By Karamata and the assumption that $\ell(n)$ is decreasing,
\begin{align*}
\|\tilde R'(z)\| & \le \Bigl\{|z-1|\sum_{j=1}^M \ell(j)+\sum_{j=M}^k j^{-1}\ell(j)\Bigr\}e^{ka}
\le \{|z-1|M\ell(M)+\ell(M)\log k\}/\ell(k)^\frac12 \\ &
\le \{|z-1|M\ell(M)+\ell(M)\}/\ell(k)^\frac32
\end{align*}
so taking $M\approx 1/(z-1)$ yields
$\|\tilde R'(z)\|\ll \ell(1/|z-1|) \ell(k)^{-\frac32}$ on $\D_a$.

Since $\ell(n)/n$ is summable,
it follows that $\ell(1/\theta)(1/\theta)$ is integrable.
Hence we can argue as in the proof of Theorem~\ref{thm-main} to
deduce that $|\rho'(n)|\ll \ell(k)^{-\frac32}e^{-na}$ and so 
\[
|\rho(n)| \ll \sum_{j>k}\ell(j)/j+n\ell(k)/k+ \ell(k)^{-\frac32}
e^{-\frac12 nk^{-1}\log(1/\ell(k))}.
\]
Define $\tilde\ell(n)=\sum_{j=n}^\infty j^{-1}\ell(j)$.
 By Karamata, $\ell(n)=o(\tilde\ell(n))$.
Taking $n=5k$, we obtain the upper bound $|\rho(n)|\ll\tilde\ell(n)$.
\end{example}

\begin{example}[Stretched exponential sequences]
\label{ex-stretch}
We consider the case $\mu(\varphi>n)\ll e^{-cn^\gamma}$,
where $\gamma\in(0,1)$ and $c>0$.    

Take $a=k^{-1}(ck^\gamma-(1+\epsilon)\log k)$.
Since $a(k)$ is eventually decreasing,
we can replace the $e^{ja(k)}$ factor in $S_q(k,a)$ by
$e^{ja(j)}$.  Then 
a calculation shows that 
$S_q(k,a)\ll 1+k^{q-\epsilon}$ for all $q\in(0,1)$.
In particular, $a^rS_r(k,a)\to0$ for $r\in(0,\epsilon/\gamma)$.
Taking $n=k$ we obtain 
$|\rho(n)|\ll n^{1+\epsilon}e^{-cn^\gamma}$ for any $\epsilon>0$.   

Even in the special setting of Young towers,
this is stronger than estimates obtained by coupling~\cite{Young99} or 
cones~\cite{Maume01a}.  
However for Young towers,
Gou\"ezel~\cite{GouezelPhD} obtains the optimal estimate 
$n^{1-\gamma}e^{-cn^\gamma}$.
In a future paper, we show how to recover Gou\"ezel's result by elementary
arguments.

The method described here works more generally for the case
$\mu(\varphi>n)\ll e^{-g(n)}$ where $g(n)$ is an increasing sequence
satisfying $g(n)=O(n^{1-\epsilon})$ for some $\epsilon>0$
and such that $a(k)=k^{-1}(g(k)-(1+\epsilon)\log k)$ is eventually
decreasing for some $\epsilon>0$.   Then we obtain
$|\rho(n)|\ll n^{1+\epsilon}e^{-g(n)}$ for any $\epsilon>0$.
\end{example}

\subsection{AFN maps}
\label{sec-AFN}

As mentioned in the introduction, the AFN maps studied 
by~\cite{Zweimuller98} have an excellent inducing scheme with 
standard function space being the space $\BV(Y)$
of observables of bounded
variations.  Unfortunately, the corresponding space $\BV(X)$ is not 
exchangeable.

Instead we take $\mathcal{B}(Y)$ to be the space of
piecewise bounded variation functions $v:Y\to\R$ with norm 
$\|v\|=\sup_{a\in\alpha}\|1_av\|_{\BV}$.
Let $\mathcal{B}(X)$ be the space of 
piecewise bounded variation functions $v:X\to\R$ with norm 
$\|v\|=\sup_{a\in\alpha,0\le\ell\le \phi(a)-1}\|1_a\,v\circ T^\ell\|_{\BV}$.
Since $T^\ell$ restricted to $a$ is a homeomorphism
for $0\le\ell\le \phi(a)-1$, it is immediate that $\mathcal{B}(X)$ is exchangeable.

It remains to show that $F:Y\to Y$ is an excellent
inducing scheme relative to $\mathcal{B}(Y)$.  
The details are standard, so we sketch the argument.
Let $\hat R$ denote the transfer operator with respect to Lebesgue measure.
Then it follows from~\cite{Rychlik83} and~\cite[Appendix]{Zweimuller98} that 
$\hat R:\BV(Y)\to \BV(Y)$ is bounded and quasicompact, so
there exist
constants $C>0$, $\tau\in(0,1)$ such that $\|\hat R^nv\|_{\BV}\le C\tau^n
\|v\|_{\BV}$ for all $n\ge1$ and all $v\in\BV(Y)$ with
$\int_Yv\,d\mu=0$.
Moreover, $d\mu=h\,dy$ where the density $h\in L^1(Y)$ satisfies 
$h,h^{-1}\in\BV(Y)$.   Hence $R=h^{-1}\hat Rh$ inherits the quasicompactness
on $\BV(Y)$ verifying (H2)(i).
Following~\cite{ADSZ04} (see for example~\cite[Subsection 11.3]{MTsub}), it is possible to extend this analysis to
$R(z)$ for all $z\in\bar\D$ and to verify (H2)(ii).
Further, $F$ has good distortion properties, so (H1) is easily verified.
Hence $F$ is excellent relative to $\BV(Y)$.

To prove excellence relative to $\mathcal{B}(Y)$, we note that
$\hat R:\mathcal{B}(Y)\to \BV(Y)$
is bounded and hence defines a bounded operator on $\mathcal{B}(Y)$.
(This is identical to the argument for $\BV(Y)$ since
$F$ satisfies a strong Rychlik condition~\cite[Condition (R), page 53]{ADSZ04}.)
Hence it is immediate from the results on $\BV(Y)$ that there exist
constants $C>0$, $\tau\in(0,1)$ such that $\|R^nv\|\le C\tau^n
\|v\|$ for all $n\ge1$ and all $v\in\mathcal{B}(Y)$ with
$\int_Yv\,d\mu=0$, verifying (H2)(i). The other properties are inherited
from $\BV(Y)$ similarly.

\appendix
\label{sec-trunc}
\section{Details for the truncation error}

In this appendix, we give the details for the truncation error~\eqref{eq-trunc}.
A similar result was proved in~\cite{M09} in a slightly more complicated
situation.   We give the details mainly for completeness and also because we
obtain a slightly improved formula (though the improvement is never used).

In particular, we use a slightly better splitting for $\Delta$, namely
$\Delta=\Delta'\dot\cup \Delta_{\rm trunc}$
where $\Delta_{\rm trunc}=\{(y,\ell)\in\Delta:\ell>k\}$.

\begin{prop} \label{prop-r}
\begin{itemize}
\item[(i)] $\bar\varphi-\bar{\varphi'}=\sum_{j>k}\mu(\varphi>j)$.
\item[(ii)]
$\mu_\Delta(\Delta_{\rm trunc})=(1/\bar\varphi)\sum_{j>k}\mu(\varphi>j)$.
\end{itemize}
\end{prop}

\begin{proof}
This is a standard computation.
\end{proof}

\begin{prop} \label{prop-En}
For $n\ge1$, define
\[
E_n=\{x\in\Delta':f^jx\in\Delta_{\rm trunc} \enspace\text{for at least
one $j\in\{1,\dots,n$\}}\}.
\]
Then $\mu_\Delta(E_n)\le n\mu(\varphi>k)\}$.
\end{prop}

\begin{proof}
Write $E_n$ as the disjoint union $E_n=
\bigcup_{j=1}^n G_j$ where
\[
G_j=\{
f^ix\in\Delta'\enspace\text{for $i\in\{0,1,\dots,j-1$\}}
\enspace\text{and}\enspace
f^jx\in\Delta_{\rm trunc}
\}.
\]
It follows from the definition that if $x\in G_j$, then $f^jx\in\Delta_{k+1}$
where $\Delta_{k+1}=\{(y,k+1):\varphi(y)>k\}$(the $(k+1)$'th level
of the tower).
Hence $\mu_\Delta(G_j)\le \mu_\Delta(f^{-j}(\Delta_{k+1})) =
\mu_\Delta(\Delta_{k+1})=(1/\bar\varphi)\mu(\varphi>k)$.
\end{proof}

\begin{cor} \label{cor-trunc}
Suppose that $v,w:\Delta\to\R$ lie in $L^\infty$.
Then for all $k,n\ge1$,
\[
\SMALL|\rho(n)-\rho'(n)|\le 
C|v|_\infty|w|_\infty \{\sum_{n>j}\mu(\varphi>j)\,+\, n\mu(\varphi>k)\}.
\]
\end{cor}

\begin{proof}
First we estimate $S={\SMALL\int}_\Delta v\,w\circ f^n\,d\mu_\Delta-
{\SMALL\int}_{\Delta'} v\,w\circ f'^n\,d\mu_{\Delta'}$.  Write
\begin{align*}
S & = 
 {\SMALL\int_{\Delta_{\rm trunc}}} v\,w\circ f^n\,d\mu_\Delta +
 \Bigl({\SMALL\int_{\Delta'}} v\,w\circ f^n\,d\mu_{\Delta} -
 {\SMALL\int_{\Delta'}} v\,w\circ f'^n\,d\mu_{\Delta} \Bigr) \\
& \qquad \qquad
 +\Bigl(
 {\SMALL\int_{\Delta'}} v\,w\circ f'^n\,d\mu_{\Delta}- 
 {\SMALL\int_{\Delta'}} v\,w\circ f'^n\,d\mu_{\Delta'}\Bigr) \\
 & = I + II + III.
\end{align*}

Now $|I|\le |v|_\infty|w|_\infty\mu_\Delta(\Delta_{\rm trunc})$.
Note that $\mu_\Delta|\Delta'=(\bar\varphi'/\bar\varphi)\mu_{\Delta'}|\Delta'$ and so
$III=(\bar\varphi'/\bar\varphi-1)\int_{\Delta'} v\,w\circ f^n\,d\mu_{\Delta'}$.
Hence $|III|\le (1/\bar\varphi)(\bar\varphi-\bar\varphi')|v|_\infty|w|_\infty$.
Next, $|II|\le 2|v|_\infty|w|_\infty\mu_{\Delta}(\Delta'\cap\{f^n\neq f'^n\})
\le 2|v|_\infty|w|_\infty\mu_{\Delta}(E_n)$.
Combining these, we obtain
\begin{align*}
|S| & \le 
|v|_\infty |w|_\infty \{\mu_\Delta(\Delta_{\rm trunc})+|\bar\varphi-\bar\varphi'|+2\mu_\Delta(E_n)\} \\ &
\le
C|v|_\infty |w|_\infty \{\sum_{n>j}\mu(\varphi>j)\,+\, n\mu(\varphi>k)\}
\end{align*}
by Propositions~\ref{prop-r} and~\ref{prop-En}.

A similar (but simpler) calculation shows that
\[
\SMALL \bigl|\int_\Delta v\,d\mu_\Delta \int_\Delta w\,d\mu_\Delta-
\int_{\Delta'} v\,d\mu_{\Delta'} \int_{\Delta'} w\,d\mu_{\Delta'}\bigr|
\le C|v|_\infty|w|_\infty \sum_{j>k}\mu(\varphi>j),
\]
and the result follows.
\end{proof}

\section{Nonuniformly hyperbolic systems}
\label{sec-NUH}

In this appendix, we show how our main results for nonuniformly expanding
maps extend to nonuniformly hyperbolic maps modelled by
Young towers~\cite{Young98,Young99}.  Even in the case of polynomial tails, this result has been missing from the literature.   
(In the case of exponential tails, Young~\cite{Young98} explicitly considers both the nonuniformly expanding and nonuniformly hyperbolic situations, but the 
subexponential tail paper~\cite{Young99} is set entirely in the nonuniformly expanding framework.)

A method for passing from nonuniformly expanding maps to nonuniformly hyperbolic systems with subexponential tails
was shown to one of us by S\'ebastien Gou\"ezel~\cite{GouezelPC} based on ideas
in~\cite{ChazottesGouezel}.  Here, we combine these ideas with
dynamical truncation.

Let $T:M\to M$ be a diffeomorphism (possibly with singularities) defined on a
Riemannian manifold $(M,d)$.   
Fix a subset $Y\subset M$.  It is assumed that there is a ``product structure'':
namely a family of ``stable disks'' $\{W^s\}$ that are disjoint and cover $Y$,
and a family of ``unstable disks'' $\{W^u\}$ that are disjoint and cover $Y$.
Each stable disk intersects each unstable disk in precisely one point.
The stable and unstable disks containing $y$ are labelled $W^s(y)$ and $W^u(y)$.

\begin{itemize}
\item[(P1)]  There is a partition $\{Y_j\}$ of $Y$ and
integers $\varphi_j\ge1$ such that $T^{\varphi_j}(W^s(y))\subset W^s(T^{\varphi_j}y)$
for all $y\in Y_j$.
\end{itemize}

Define the return time function $\varphi:Y\to\Z^+$ by $\varphi|_{Y_j}=\varphi_j$
and the induced map
$F:Y \to Y$ by $F(y)=T^{\varphi(y)} (y)$.

Let $s$ denote the {\em separation time} with respect to the map $F:Y\to Y$.  
That is, if
$y,z\in Y$, then $s(y,z)$ is the least integer $n\ge0$ such that $F^n x$, $F^ny$ lie in distinct partition elements of $Y$.    
\begin{itemize}
\item[(P2)]
There exist constants $C\ge1$, $\gamma_0\in(0,1)$ such that
\begin{itemize}
\item[(i)]  If $z\in W^s(y)$, then $d(F^ny,F^nz)\le C\gamma_0^n$, 
\item[(ii)]  If $z\in W^u(y)$, then $d(F^ny,F^nz)\le C\gamma_0^{s(y,z)-n}$,
\item[(iii)] If $y,z\in Y$, then $d(T^jy,T^jz)\le Cd(y,z)$
for all $0\le j<\min\{\varphi(y),\varphi(z)\}$.
\end{itemize}
\end{itemize}

Let $\bar Y=Y/\sim$ where $y\sim z$ if $y\in W^s(z)$ and
define the partition $\{\bar Y_j\}$ of $\bar Y$.
We obtain a well-defined return time function  $\varphi:\bar Y\to\Z^+$ and
induced map $\bar F:\bar Y\to\bar Y$.

\begin{itemize}
\item[(P3)]  The map $\bar F:\bar Y\to\bar Y$ and partition $\{\bar Y_j\}$ separate points in $\bar Y$.
(It follows that $d_\theta(y,z)=\theta^{s(y,z)}$ defines
a metric on $\bar Y$ for each $\theta\in(0,1)$.)
\item[(P4)]  There exists an invariant ergodic probability measure $\mu_{\bar Y}$ on $\bar Y$ such that $F:\bar Y\to\bar Y$ is a Gibbs-Markov map
in the sense of Example~\ref{ex-Young} and $\varphi:\bar Y\to\Z^+$ is integrable.
\end{itemize}

From (P4), a standard construction leads to 
an invariant probability measure 
$\mu_Y$ on $Y$ such that $\bar\pi_*\mu_Y=\mu_{\bar Y}$ where $\bar\pi:Y\to\bar Y$ is the quotient map.
There is also a standard method to pass from $\mu_Y$ to a measure
$\nu$ on $M$ which we recall now.    As in Section~\ref{sec-tower},
starting from $\bar F:\bar Y\to \bar Y$ and $\varphi:\bar Y\to\Z^+$, we can form a {\em quotient tower} $\bar \Delta$
and a {\em quotient tower map} $\bar f:\bar \Delta\to\bar \Delta$ such that $\bar F=\bar f^\varphi:\bar Y\to \bar Y$ is a
{\em first} return map  for $\bar f$.   Then
$\mu_{\bar\Delta}=(\mu_{\bar Y}\times{\rm counting})/\int_{\bar Y}\varphi\,d\mu_{\bar Y}$ is an
$\bar f$-invariant probability measure on $\bar \Delta$.   

Similarly, starting from $F:Y\to Y$ and $\varphi:Y\to\Z^+$, we can form a tower $\Delta$
and tower map $f:\Delta\to\Delta$ such that $F=f^\varphi:Y\to Y$ is a
{\em first} return map  for $f$.   Again,
$\mu_\Delta=(\mu\times{\rm counting})/\int_Y\varphi\,d\mu_Y$ is an
$f$-invariant probability measure on $\Delta$.   Define the semiconjugacy
$\pi:\Delta\to M$, $\pi(y,\ell)=T^\ell y$.  Then
$\nu=\pi_*\mu_\Delta$ is the desired measure on $M$.
(We omit the additional assumptions in Young~\cite{Young98}
that guarantee that $\nu$ is an SRB measure.   The results in this appendix do
not rely on this property.)

Let $v_0,w_0:M\to\R$ be $C^\eta$ observables ($\eta\in(0,1)$) and
define the correlation function
$\rho_{v_0,w_0}(n)  ={\SMALL\int}_M v_0\,w_0\circ T^n\,d\nu-{\SMALL\int}_M v_0\,d\nu{\SMALL\int}_M w_0\,d\nu$.   We obtain the following analogue of
Theorem~\ref{thm-main}(i).

\begin{thm} \label{thm-NUH}
 Let $a=a(k)$ be such that $\lim_{k\to\infty}a^rS_r(k,a)=0$ for some 
$r\in(0,1]$. 
Let $q\in(0,1]$.
 Then there exists $C>0$, $k_0\ge1$ such that 
 \begin{align*}
&  |\rho_{v_0,w_0}(n)|  \le 
 C|v_0|_\infty |w_0|_\infty \Bigl(\sum_{j>k}\mu(\varphi>j)+n\mu(\varphi>k)\Bigr)
+C\|v_0\|_{C^\eta}\|w_0\|_{C^\eta} S_q(k,a)e^{-\frac12 na},
 \end{align*}
 for all $v_0,w_0\in C^\eta(M)$, $n\ge k\ge k_0$.
 \end{thm}

\begin{rmk}  

Thus, we obtain identical results for the nonuniformly hyperbolic case
as for the nonuniformly expanding case,
except that $a(k)$ is replaced by $\frac12 a(k)$.
In particular, we obtain optimal results for polynomial
decay, and more generally for polynomially decreasing sequences.
In addition, Corollary~\ref{cor-poly} and Theorem~\ref{thm-slow} remain valid.
The only result that deteriorates in passing to the nonuniformly hyperbolic 
setting is
the estimate for stretched exponential decay in Example~\ref{ex-stretch}
where we obtain the decay rate $O(n^{1+\epsilon}e^{-\frac12 cn^{\gamma}})$.
%
\end{rmk}

In the remainder of this appendix, we prove Theorem~\ref{thm-NUH}.

\paragraph{Decay of correlations on $\Delta$}
Given $C^\eta$ observables $v_0,w_0:M\to\R$, let
$v=v_0\circ \pi,w=w_0\circ \pi:\Delta\to\R$ be the lifted observables.
Since $\pi:\Delta\to M$ is a semiconjugacy and $\nu=\pi_*\mu_\Delta$,
to prove Theorem~\ref{thm-NUH} it is equivalent to estimate
the correlation function
$\rho_{v,w}(n) 
={\SMALL\int}_\Delta v\,w\circ f^n\,d\mu_\Delta-{\SMALL\int}_\Delta v\,d\mu_\Delta{\SMALL\int}_\Delta w\,d\mu_\Delta$.

\paragraph{Dynamical truncation} 
For $k\ge1$ fixed, set $\varphi'=\min\{\varphi,k\}$ to form a truncated
tower map $f':\Delta'\to\Delta'$ (with invariant probability measure
$\mu_{\Delta'}$).   
Let
$\rho_{v,w}'(n)={\SMALL\int}_{\Delta'} v\,w\circ f'^n\,d\mu'
-{\SMALL\int}_{\Delta'} v\,d\mu'{\SMALL\int}_{\Delta'} w\,d\mu'$.
We obtain the same truncation error~\eqref{eq-trunc}
as in the nonuniformly hyperbolic case.
Hence it remains to prove under the assumptions of Theorem~\ref{thm-NUH} that
\begin{align} \label{eq-NUH}
|\rho_{v,w}'(n)|\le 
C\|v_0\|_{C^\eta}\|w_0\|_{C^\eta} S_q(k,a)e^{-\frac12 na}.
\end{align}

\paragraph{Quotient towers and function spaces}
We use the separation time for $F:Y\to Y$ to define a separation time on 
$\Delta$:
define $s((y,\ell),(z,m))=s(y,z)$
if $\ell=m$ and $0$ otherwise.
This drops down to separation times $s$ on $\bar\Delta$ and $\bar Y$.

Given $\theta\in(0,1)$, we define the symbolic metric $d_\theta$ on
$\bar\Delta$ by setting $d_\theta(p,q)=\theta^{s(p,q)}$.
In particular, $d_\theta$ is a metric on $\bar Y$.
Define the spaces $\mathcal{B}(\bar \Delta)$,
$\mathcal{B}(\bar Y)$
of $d_\theta$-Lipschitz observables on $\bar\Delta$ and $\bar Y$ respectively.
Then 
$\mathcal{B}(\bar Y)$ satisfies our
main hypotheses (H1) and (H2),   
and $\mathcal{B}(\bar\Delta)$ is exchangeable.

\paragraph{Nonuniform expansion/contraction}

Recall that $\pi:\Delta'\to M$ denotes
the projection $\pi(y,\ell)=T^\ell y$.
For $p=(x,\ell),q=(y,\ell)\in\Delta'$, we write $q\in W^s(p)$ if $y\in W^s(x)$ and
$q\in W^u(p)$ if $y\in W^u(x)$.
Conditions (P2) translate as follows.
\begin{itemize}
\item[(P2$'$)]
There exist constants $C\ge1$, $\gamma_0\in(0,1)$ such that
for all $p,q\in \Delta'$, $n\ge1$,
\begin{itemize}
\item[(i)]  If $q\in W^s(p)$, then $d(\pi f'^np,\pi f'^nq)\le C\gamma_0^{\psi'_n(p)}$, and
\item[(ii)]  If $q\in W^u(p)$, then $d(\pi f'^np,\pi f'^nq)\le C\gamma_0^{s(p,q)-\psi'_n(p)}$, 
\end{itemize}
\end{itemize}
where $\psi'_n(p)=\#\{j=0,\dots,n-1:f'^jp\in Y\}$ is the number of returns of $p$ to $Y$ by time $n$.

\begin{rmk}   These properties can be defined at the level of the nontruncated 
tower $\Delta$.  Since $F$ is independent of $k$, the constants $\gamma_0$ and $C$
are unchanged by truncation and hence are independent of $k$.
Also, $s(p,q)$ is independent of $k$.
Of course, $\psi'_n(p)$ decreases monotonically with $k$, and we have the
estimate $n/k\le \psi'_n\le n$.
\end{rmk}

\begin{prop} \label{prop-W}
$d(\pi f'^n p,\pi f'^n q)\le C\gamma_0^{\min\{\psi'_n(p),s(p,q)-\psi'_n(p)\}}$ for all
$p,q\in\Delta$, $n\ge1$.
\end{prop}

\begin{proof}  This is immediate from conditions (P2$'$) and the product 
structure on $Y$.
\end{proof}

\paragraph{Approximation of observables}
Let  $v=v_0\circ\pi:\Delta'\to\R$ be the lift of a $C^\eta$ observable
$v_0:M\to\R$.   For each $n\ge1$, define $\tilde v_n:\Delta'\to\R$,
\[
\tilde v_n(p)=\inf\{v(f'^nq):s(p,q)\ge 2\psi'_n(p)\}.
\]
We list some standard properties of $\tilde v_n$.
Recall that $L'$ is the transfer operator corresponding to $\bar f':\bar\Delta\to\bar\Delta$.

\begin{prop} \label{prop-tildev}
The function $\tilde v_n$ lies in $L^\infty(\Delta')$ and projects down to a Lipschitz 
observable $\bar v_n:\bar\Delta'\to\R$.  Moreover, setting $\gamma=\gamma_0^\eta$
and $\theta=\gamma^{\frac12}$,
\begin{itemize}
\item[(a)] $|\bar v_n|_\infty=|\tilde v_n|_\infty\le |v_0|_\infty$.
\item[(b)]  
$|v\circ f'^n(p)-\tilde v_n(p)|_\infty \le C\|v_0\|_{C^\eta}\gamma^{\psi'_n(p)}$ for $p\in\Delta'$.
\item[(c)]  $\|L'^n\bar v_n\|_\theta\le C\|v_0\|_{C^\eta}$.  
\end{itemize}
\end{prop}

\begin{proof}
If $s(p,q)\ge 2\psi'_n(p)$, then $\tilde v_n(p)=\tilde v_n(q)$.
It follows that $\tilde v_n$ is piecewise constant on a measurable partition 
of $\Delta'$, and hence is measurable, and that $\bar v_n$ is well-defined.
Part (a) is immediate.   

Recall that $v=v_0\circ\pi$ where $v_0:M\to\R$ is $C^\eta$.
Let $p\in\Delta'$.
By Proposition~\ref{prop-W} and the definition of $\tilde v_n$,
\begin{align*}
|v\circ f'^n(p)-\tilde v_n(p)| &=|v_0(\pi f'^np)-v_0(\pi f'^nq)|
 \le \|v_0\|_{C^\eta}d(\pi f'^np,\pi f'^nq)^\eta
\\ & \le C'\gamma^{\min\{\psi'_n(p),s(p,q)-\psi'_n(p)\}},
\end{align*}
where $q$ is such that $s(p,q)\ge 2\psi'_n(p)$.
In particular, $s(p,q)-\psi'_n(p)\ge \psi'_n(p)$, so we obtain part~(b).

To prove (c), recall that
$(L'^n\bar v_n)(\bar p)=\sum_{\bar f'^n\bar q=\bar p}g_n(\bar q)\bar v_n(\bar q)$
where $g$ is the weight function.
It is immediate that $|L'^n\bar v_n|_\infty \le |\bar v_n|_\infty\le |v_0|_\infty$.
Write
\begin{align} \label{eq-partc} \nonumber
(L'^n\bar v_n)(\bar p_1)- (L'^n\bar v_n)(\bar p_2) & =
\sum_{\bar f'^n\bar q_1=\bar p_1}g_n(\bar q_1)(\bar v_n(\bar q_1)-\bar v_n(\bar q_2)) \\ &\qquad +
\sum_{\bar f'^n\bar q_1=\bar p_1}(g_n(\bar q_1)-g_n(\bar q_2))\bar v_n(\bar q_2).
\end{align}
Naturally, we pair up preimages so that
$s(\bar q_1,\bar q_2)=\psi'_n(\bar q_1)+s(\bar p_1,\bar p_2)$.
We then choose $q_1,q_2\in\Delta'$ that project onto $\bar q_1,\bar q_2\in\bar\Delta'$, so
\begin{align} \label{eq-pairs}
s(q_1,q_2)=s(\bar q_1,\bar q_2)=\psi'_n(\bar q_1)+s(\bar p_1,\bar p_2).
\end{align}

By standard arguments, the second term in~\eqref{eq-partc}
contributes $C|v_0|_\infty$ 
to the norm of $L'^n\bar v_n$.   
We claim that $|\bar v_n(\bar q_1)-\bar v_n(\bar q_2)|
\le C\|v_0\|_{C^\eta} \gamma^{\frac12 s(\bar p_1,\bar p_2)}$.
Taking $\theta=\gamma^{\frac12}$, 
it then follows that the first term in~\eqref{eq-partc} contributes $C\|v_0\|_{C^\eta}$
to the norm of $L'^n\bar v_n$.   

It remains to verify the claim.
Write
\[
\bar v_n(\bar q_1)-\bar v_n(\bar q_2)=
v\circ f'^n(\hat q_1 )-v\circ f'^n(\hat q_2),
\]
where 
$\hat q_1,\hat q_2\in\Delta'$ satisfy
\begin{align} \label{eq-qhat}
s(\hat q_j,q_j)\ge 2\psi'_n(\bar q_j).
\end{align}
Moreover, $\bar v_n(\bar q_1)=\bar v_n(\bar q_2)$ if $s(q_1,q_2)\ge 2\psi'_n(\bar q_1)$, so we may suppose without loss that 
\begin{align} \label{eq-wlog}
s(q_1,q_2)\le 2\psi'_n(\bar q_1).
\end{align}
As in part (b),
\begin{align} \label{eq-min}
|v\circ f'^n(\hat q_1)-v\circ f'^n(\hat q_2)| 
\le C\|v_0\|_{C^\eta}\gamma^{\min\{\psi'_n(\hat q_1),s(\hat q_1,\hat q_2)-\psi'_n(\hat q_1)\}}.
\end{align}
By~\eqref{eq-pairs} and~\eqref{eq-qhat},
\[
s(\hat q_1,\hat q_2)-\psi'_n(\bar q_1)
\ge \min\{s(q_1,q_2),s(\hat q_1,q_1),s(\hat q_2,q_2)\}-\psi'_n(\bar q_1)
\ge \min\{s(\bar p_1,\bar p_2),\psi'_n(\bar q_1)\}.
\]
By~\eqref{eq-qhat} and~\eqref{eq-wlog},
\[
\SMALL \psi'_n(\hat q_1)=\psi'_n(\bar q_1)\ge\frac12 s(\bar q_1,\bar q_2)
\ge \frac12 s(\bar p_1,\bar p_2).
\]
Substituting these into~\eqref{eq-min} establishes the claim.
\end{proof}

The next property draws on ideas from~\cite[Lemma~4.4]{ChazottesGouezel}.
\begin{lemma} \label{lem-tilde}
Suppose that $a =a(k)$ satisfies
$\lim_{k\to\infty}aS_0(k,a)=0$.  Let $r\in(0,1]$.
There exists $k_0\ge1$ such that
\[
|v\circ f'^n-\tilde v_n|_1 \ll (1+a^rS_r(k,a))^2e^{-na}\|v_0\|_{C^\eta},
\]
for all $n\ge k\ge k_0$.
\end{lemma}

\begin{proof}
By Proposition~\ref{prop-tildev}(b), $|v\circ f'^n(p)-\tilde v_n(p)|\ll \gamma^{\psi'_n(p)}\|v_0\|_{C^\eta}$.
Note that $\psi'_n=\sum_{j=0}^{n-1}\psi\circ f'^j$ where $\psi=1_Y$.    
We have
\[
|v\circ f'^n-\tilde v_n|_1/\|v_0\|_{C^\eta} \ll \int_{\Delta'} \gamma^{\psi'_n}\,d\mu'=
\int_{\bar\Delta'} \gamma^{\psi'_n}\,d\mu'=
\int_{\bar\Delta'} L'^n\gamma^{\psi'_n}\,d\mu'=
\int_{\bar\Delta'} L'^n_\gamma 1\,d\mu.
\]
where $L'_\gamma$ is the twisted transfer operator $L'_\gamma v=L'(\gamma^\psi v)$.

We estimate $L'^n_\gamma$ using truncated renewal operators.  Define
\begin{alignat*}{2}
 T'_{n,\gamma}& =1_Y L'^n_\gamma 1_Y,   
& \qquad T'_\gamma(z)& =\sum_{n=0}^\infty T'_{n,\gamma}z^n,
\\
 R'_{n,\gamma}& =1_Y L'^n_\gamma 1_{\{\varphi'=n\}},
& \qquad R'_\gamma(z)& =\sum_{n=1}^\infty R'_{n,\gamma}z^n .
\end{alignat*}
Then the renewal equation takes the form
$T'_\gamma(z)=(I-R'_\gamma(z))^{-1}$, for $z\in\D$.
Throughout, $\gamma\in(0,1)$ is fixed.

Next, we observe that
\[
R'_{n,\gamma}v= L'^n_\gamma( 1_{\{\varphi'=n\}}v)=L'^n(\gamma^{\psi'_n} 1_{\{\varphi'=n\}}v)
=\gamma R'_nv.
\]
In particular, $R'_\gamma(z)=\gamma R'(z)$ for $z\in\C$.
Similarly, we can define $R_\gamma(z)$ and deduce that 
$R_\gamma(z)=\gamma R(z)$, $z\in\bar\D$.
Hence, the spectral radius of $R_\gamma(z)$ is at
most $\gamma$ for all $z\in\bar\D$.   It follows that
$\sup_{z\in\bar\D}\|(I-R_\gamma(z))^{-1}\|<\infty$ for $z\in\bar\D$.
We proceed as in the proof 
of Proposition~\ref{prop-first} to deduce that for $k\ge k_0$, first
$\sup_{z\in\bar\D}\|(I-R'_\gamma(z))^{-1}\|<\infty$, and then that
\[
\sup_{z\in\D_a}\|T'_\gamma(z)\|=
\sup_{z\in\D_a}\|(I-R'_\gamma(z))^{-1}\|\ll 1.
\]

The relation $L'(z)=A'(z)T'(z)D'(z)+E'(z)$ from Section~\ref{sec-L}
 holds in the presence of $\gamma$ (with the obvious definitions) and it is immediate that
\[
A'_\gamma(z)=\gamma A'(z), \quad D'_\gamma(z)=D'(z), \quad E'_\gamma(z)=E'(z).
\]
In particular $E'_\gamma$ is a polynomial of degree at most $k-1$.
By Corollary~\ref{cor-AD}(b,c),
$\sup_{z\in\D_a}\|A'(z)\|\ll 1+a^rS_r(k,a)$, and
$\sup_{z\in\D_a}\|D'(z)\|\ll 1+a^rS_r(k,a)$.
Hence 
$\sup_{z\in\D_a}\|L'_\gamma(z)\|\ll (1 + a^rS_r(k,a))^2$
and the result follows.
\end{proof}

\begin{rmk}   The spectral radius property for $R_\gamma(z)$ holds in
$L^1(Y)$, so it is possible to prove Lemma~\ref{lem-tilde} without passing
to the Lipschitz norm.  However, this does not seem to lead to improvements 
in our final results.
\end{rmk}

\begin{pfof}{Theorem~\ref{thm-NUH}}
Suppose without loss that $v$ is mean zero.
Let $\ell\ge1$, and write
\begin{align*}
\rho'(n)& =\int_{\Delta'}v\,w\circ f'^n\,d\mu'
=\int_{\Delta'}v\circ f'^\ell\,w\circ f'^{\ell+n}\,d\mu'
 = I_1+I_2+I_3,
\end{align*}
where
\begin{align*}
I_1 & =\int_{\Delta'}(v\circ f'^\ell-\tilde v_\ell)\,w\circ f'^{\ell+n}\,d\mu'\\
I_2 &=\int_{\Delta'}\tilde v_\ell\,(w\circ f'^{n/2}-\tilde w_{n/2})\circ f'^{\ell+n/2}\,d\mu' \\
I_3 & =\int_{\Delta'}\tilde v_\ell\,\tilde w_{n/2}\circ f'^{\ell+n/2}\,d\mu'.
\end{align*}

By Proposition~\ref{prop-tildev}(b), $|I_1|\le|v\circ f'^\ell-\tilde v_\ell|_\infty |w|_\infty
\le C|\gamma^{\psi_n'}\|v_0\|_{C^\eta}|w_0|_\infty
\le C\gamma^{\ell/k}\|v_0\|_{C^\eta}|w_0|_\infty$.
By Proposition~\ref{prop-tildev}(a) and Lemma~\ref{lem-tilde}, $|I_2|\le|\tilde v_\ell|_\infty |w\circ f'^{n/2}-\tilde w_{n/2}|_1
\ll |v_0|_\infty \|w_0\|e^{-\frac12 na(k)}$.
Assume for the moment that
$\tilde v_\ell$ is mean zero.  By Theorem~\ref{thm-main} and
Proposition~\ref{prop-tildev}(c),
\begin{align*}
|I_3|
& =\Bigl|\int_{\bar\Delta'}\bar v_\ell\,\bar w_{n/2}\circ \bar f'^{\ell+n/2}\,d\mu'\Bigr|
=\Bigl|\int_{\bar\Delta'}L'^{n/2}L'^{\ell}\bar v_\ell\,\bar w_{n/2}\,d\mu'\Bigr| \\
& \le |L'^{n/2}L'^\ell \bar v_\ell|_1 |\bar w_{n/2}|_\infty
\ll S_q(k,a)e^{-\frac12 na}\|L'^\ell \bar v_\ell\|_\theta |w_0|_\infty
 \ll S_q(k,a)e^{-\frac12 na}\|v_0\|_{C^\eta} |w_0|_\infty.
\end{align*}
In the general case where $\tilde v_\ell$ is not mean zero, we apply the 
above argument with $\tilde v_\ell$ replaced by $\tilde v_\ell-\int_{\Delta'}\tilde v_\ell\,d\mu'$, and there is an extra term bounded by 
$|\int_{\Delta'}\tilde v_\ell\,d\mu'| |w_0|_\infty$.  Since $v$ is mean zero, 
$|\int_{\Delta'}\tilde v_\ell\,d\mu'|=
|\int_{\Delta'}(\tilde v_\ell-v\circ f'^\ell)\,d\mu'|\le C\|v_0\|_{C^\eta}e^{-\ell/k}$ by
another application of Proposition~\ref{prop-tildev}(b).

Finally, $\ell$ is arbitrary, and letting $\ell\to\infty$ yields the result.
\end{pfof}

\paragraph{Acknowledgements}
The research of IM and DT was supported in part by EPSRC Grant EP/F031807/1.
We are very grateful to S\'ebastien Gou\"ezel, Stefano Luzzatto and
Sandro Vaienti for helpful discussions and encouragement.
Special thanks to
S\'ebastien Gou\"ezel for showing IM the additional ideas~\cite{ChazottesGouezel, GouezelPC} required for Theorem~\ref{thm-NUH}.

\end{document}